\documentclass[11pt]{article}
\usepackage{amssymb, amsmath}
\usepackage{amsthm}
\usepackage[margin=1in]{geometry}
\usepackage{dsfont}
\usepackage{enumerate}
\allowdisplaybreaks[1.2]
\numberwithin{equation}{section}

\newtheorem{thm}{Theorem}[section]
\newtheorem{cor}{Corollary}[section]

\newtheorem{pro}{Proposition}[section]
\newtheorem{lemma}{Lemma}[section]
\usepackage{graphicx}
\usepackage[mathscr]{euscript}
\usepackage{hyperref}

\begin{document}
	
	\markboth{T. Anitha, R. Rajkumar and Andrei Gagarin}{}
	\title{The complement of proper power graphs of finite groups}
	
	\author{T. Anitha\footnote{e-mail: {\tt tanitha.maths@gmail.com}},\ \ \
		R. Rajkumar\footnote{e-mail: {\tt rrajmaths@yahoo.co.in}}\\
		{\footnotesize Department of Mathematics, The Gandhigram Rural Institute -- Deemed to be University,}\\ \footnotesize{Gandhigram -- 624 302, Tamil Nadu, India.}\\[3mm]
Andrei Gagarin\footnote{e-mail: {\tt gagarina@cardiff.ac.uk}}\\
	{\footnotesize School of Mathematics,
		Cardiff University,}\\ \footnotesize{Cardiff, CF24 4AG,
		Wales, UK}\\[3mm]
}
	\date{}
	\maketitle

\begin{abstract}
For a finite group $G$, the proper power graph $\mathscr{P}^*(G)$ of $G$ is the graph whose vertices are non-trivial elements of $G$ and two vertices $u$ and $v$ are adjacent if and only if $u \neq v$ and $u^m=v$ or $v^m=u$ for some positive integer $m$. In this paper, we consider the complement of $\mathscr{P}^*(G)$, denoted
by ${\overline{\mathscr{P}^*(G)}}$. We classify all finite groups whose complement of proper power graphs is complete, bipartite, a path, a cycle, a star, claw-free, triangle-free, disconnected, planar, outer-planar, toroidal, or projective. Among the other results, we also determine the diameter and girth of the complement of proper power graphs of finite groups.
\end{abstract}

	\section{Introduction}
The investigation of properties of a given algebraic structure can be made by associating it with a suitable graph, and then by analyzing the graph properties using methods of graph theory. This approach have been used in a great amount of literature, for example, see \cite{Abdollahi},  \cite{behoodi}, \cite{eleo}. Moreover, there are several recent papers dealing with embeddability of graphs associated with algebraic structures on topological surfaces. For instance, see \cite{afkhami}, \cite{hung}, \cite{raj}, \cite{raj2}, \cite{raj5 }. Kelarev and Quinn \cite{Kelarev} introduced and studied the directed power graph of a semigroup. The \textit{directed power graph} of a semigroup $S$ is a digraph having the vertex set $S$, and, for $u,v\in S$, there is an arc from $u$ to $v$ if and only if $u\neq v$ and $v=u^m$ for some positive integer $m$. Later, Chakrabarty et al. \cite{Chakrabarty} defined the \textit{undirected power graph} $\mathscr{P}(G)$ of a group $G$ as an undirected graph whose vertex set is $G$, and two vertices $u$ and $v$ are adjacent if and only if $u\neq v$ and $u^m=v$ or $v^m=u$ for some positive integer $m$. Recently, several interesting results have been obtained for these graphs. For instance, see \cite{Min}, \cite{Sriparna}. Mirzargar et al. \cite{Mirzargar} investigated planarity of the undirected power graph of finite groups, and  Xuanlong Ma and Kaishun Wang \cite{Xuanlong} classified all finite groups whose undirected power graphs can be embedded on the torus.

Further, in \cite{Mogha}, Moghaddamfar et at. considered the graph $\mathscr{P}^*(G)$, which is obtained by removing the identity element from the undirected power graph $\mathscr{P}(G)$  of a given group $G$. This graph is called the \emph{undirected proper power graph of $G$}. They have studied several properties of these graphs, including the classification of finite groups whose undirected proper power graphs are one of strongly regular, bipartite, planar, or Eulerian. Later, in \cite{Doostabadi}, Doostabedi and Farroki have investigated various kinds of planarity, toroidality, and projective-planarity of these graphs. An interested reader may refer to the survey \cite{survey} for further results and open problems related to the power graph of groups and semigroups. In this paper, we consider only the undirected graphs, and, for simplicity, use the term `power graph' to refer to the undirected power graph.

In this paper, we study the properties of complement of the proper power graph of a group. For a given group $G$, the \textit{complement of the proper power graph of $G$}, denoted by ${\overline{\mathscr{P}^*(G)}}$, is a graph whose vertex set is the set of all nontrivial elements of $G$, and two vertices $u$ and $v$ are adjacent if and only if $u \neq v$, and $u^m\neq v$ and $v^n\neq u$ for any positive integers $m$, $n$; in other words $u$ and $v$ are adjacent if and only if $u \neq v$, $u\notin \left\langle v\right\rangle $ and $v\notin \left\langle u\right\rangle $.

The rest of the paper is organized as follows.
In Section 2, we provide some preliminaries and notations.
In Section 3, we classify all finite groups whose complement of proper power graph is complete, bipartite, $C_3$-free, $K_{1,3}$-free, disconnected, or having isolated vertices. Moreover, in this section, we determine the girth and diameter of the complement of proper power graphs of finite groups. In Section 4, we classify all finite groups whose complement of proper power graphs is planar, toroidal, or projective-planar. As a consequence, we classify the finite groups whose complement of proper power graphs is a path, a star, a cycle, outer-planar, or having no subgraphs $K_{1,4}$ or $K_{2,3}$.


\section{Preliminaries and notations} \label{sec:pre} In this section, we remind some concepts, notation, and results in graph theory and group theory. We follow the terminology and
notation of  \cite{harary, White} for graphs and \cite{Scott} for groups.
A graph $G$ is said to be \emph {complete} if there is an edge between every pair of its distinct vertices. $G$ is said to be \emph {$k$-partite} if the vertex set of $G$ can be partitioned to $k$ subsets, called \emph{parts} of $G$, such that no two vertices in the same subset of the partition are adjacent. A \emph {complete $k$-partite} graph, denoted by $K_{n_1, n_2, \ldots, n_k}$, is a $k$-partite graph having its parts sizes $n_1, n_2, \ldots, n_k $ such that  every vertex in each part is adjacent to all the vertices in the other parts of $K_{n_1, n_2, \ldots, n_k}$. For simplicity, we denote the complete $k$-partite  graph $K_{n, n, \cdots, n}$ by $K(k,n)$. The graph $K_{1,n}$ is called a \emph{star}. $P_n$ and $C_n$ respectively denote the \emph  {path} and \emph {cycle} on $n$ vertices. We denote  the degree of  a vertex $v$ in $G$ by $deg_G(v)$. $G$ is said to be \emph{$H$-free}, if it has no induced subgraph isomorphic to $H$.

$G$ is said to be \emph {connected} if there exists a path between any two distinct vertices in the graph; otherwise $G$ is said to be \emph {disconnected}. The distance between two vertices $u$ and $v$ of a graph, denoted by \emph{$d(u,v)$}, is the length of a shortest path between $u$ and $v$ in the graph, if such a path exists, and $\infty$ otherwise. The \emph{diameter} of a connected graph $G$ is the maximum distance between any two vertices in the graph, and is denoted by $diam(G)$.  The number of edges in a path or a cycle, is called its \emph{length}.  The \emph{girth} of $G$ is the minimum of the lengths of all cycles in $G$, and is denoted by \emph{gr$(G)$}. If $G$ is acyclic, that is, if $G$ has no cycles, then we write gr$(G)=\infty$.  The \textit{complement} $\overline{G}$ of $G$ is a
graph, which has the vertices of $G$ as its vertex set, and two vertices in $\overline{G}$ are adjacent if and only if they are not adjacent in $G$.
Given two simple graphs, $G_1 =(V_1, E_1)$ and $G_2 = (V_2, E_2)$,  their \textit{union}, denoted by $G_1 \cup G_2$, is a graph with
the vertex set $V_1 \cup V_2$ and the edge set $E_1 \cup E_2$. Their \textit{join}, denoted by $G_1 + G_2$, is a graph having  $G_1 \cup G_2$ together with all the edges
joining points of $V_1$ to points of $V_2$.

A graph is said to be \textit{embeddable} on a topological surface if it can be drawn on the surface in such a way that no two edges cross.  The \emph{orientable genus} or \emph{genus} of a graph $G$, denoted by $\gamma(G)$, is the smallest non-negative integer $n$
such that  $G$ can be embedded on the sphere with $n$ handles. $G$ is said to be \emph{planar} or \emph{toroidal}, respectively, when $\gamma(G)$ is either $0$ or $1$. A planar graph $G$ is said to be \emph{outer-planar} if it can be drawn in the plane with all its vertices lying on the same face.

A \emph{crosscap} is a circle (on a surface) such that all its pairs of opposite points are identified, and the interior of this circle  is removed.  The \emph{nonorientable} genus of $G$, denoted by $\overline{\gamma}(G)$, is the smallest integer $k$ such that $G$ can be embedded on the sphere with $k$ crosscaps. $G$ is said to be \emph{projective} or \emph{projective-planar} if $\overline{\gamma}(G)=1$.   Clearly, if $G'$ is a subgraph of $G$, then $\gamma(G')\leq\gamma (G)$ and $\overline{\gamma}(G')\leq\overline{\gamma} (G)$.

Let $G$ be a group. The order of an element $x$ in $G$ is denoted by $o(x)$. For a positive integer $n$, $\varphi(n)$ denotes the Euler's totient function of $n$. For any integer $n\geq 3$, the dihedral group of order $2n$ is given by $D_{2n}=\left\langle a,b| a^n = b^2 = e, ab = ba^{-1} \right\rangle  $. For any integer $n \geq 2$, the quarternion group of order $4n$ is given by $Q_{4n}=\left\langle a,b| a^{2n}=b^4=1,b^2=a^n, ab=ba^{-1}  \right\rangle $.  For any $\alpha\geq 3$ and a prime $p$, the modular group of order $p^{\alpha}$ is given by $M_{p^{\alpha}}=\langle a,b| a^{p^{\alpha-1}}=b^p=1,~bab^{-1}=a^{p^{\alpha-2}}+1 \rangle $. Throughout this paper, $p$, $q$ denotes distinct prime numbers.

The following results are used in the subsequent sections.
\begin{thm}(\cite[Theorem 6.6]{White})
	A graph $G$ is planar if and only if $G$
	contains no subgraphs homeomorphic to $K_5$ or $K_{3,3}$.
\end{thm}

\begin{thm}{(\cite[Theorems 6.37, 6.38, 11.19, 11.23]{White})}\label{gamma(K_n)}
	\begin{enumerate}[{\normalfont (1)}]
		\item $\gamma(K_n)=\left \lceil \frac {(n-3)(n-4)}{12} \right \rceil $ , $n\geq 3$
		\item $\gamma(K_{m,n})=\left \lceil\frac{(m-2)(n-2)}{4} \right \rceil$ , $n,m\geq 2$
		\item $\overline{\gamma}(K_n)=\left \lceil \frac {(n-3)(n-4)}{6} \right \rceil$ , $n\geq 3$, $n\neq 7$; $\gamma(K_n)=3$ if $n=7$.
		\item $\overline{\gamma}(K_{m,n})=\left \lceil\frac{(m-2)(n-2)}{2} \right \rceil$ , $n,m\geq 2$.
	\end{enumerate}
	As a consequence,  $\gamma(K_n)>1$ for $n\geq 8$, $\overline{\gamma}(K_n)>1$ for $n\geq 7$,
	$\gamma(K_{m,n})>1$ if either $m\geq 4,~n\geq 5$ or $m\geq 3,~n\geq 7$, and $\overline{\gamma}(K_{m,n})>1$  if either $m\geq 3$, $n\geq 5$ or $m=n=4$.
\end{thm}
\begin{thm}\label{pre}
	\begin{enumerate}[{\normalfont (i)}]
		\item ({\cite[p.129]{millar}})
		The number of non-cyclic subgroup of order $p^{\alpha}$ in any non-cyclic group of order $p^m$ is of the form $1+kp$ whenever $1<\alpha<m$ and $p>2$.
		\item ({\cite[Proposition 1.3]{Scott}})
		If $G$ is a $p$-group of order $p^n$, and it has a unique subgroup of order $p^m,~1<m\leq n$, then $G$ is cyclic or $m=1$ and $p=2$, $G\cong Q_{2^\alpha}$.
		\item ({\cite[Theorem IV, p.129]{Burnside}})
		If $G$ is a $p$-group of order $p^n$, then the number of subgroups of order $p^s,~1\leq s\leq n$ is congruent to 1 (mod~$p$).
	\end{enumerate}
	
\end{thm}


\section{Some results on the complement of proper power graphs of groups}
\begin{thm}\label{complete}
	Let $G$ be a finite group. Then $\overline{\mathscr{P}^*(G)}$ is complete if and only if $G\cong \mathbb Z_{2}^m$, $m\geq 1$.
\end{thm}
\begin{proof}
	Assume that $\overline{\mathscr{P}^*(G)}$ is complete. Then every element of $G$ is of order 2. For, if $G$ contains an element $x$ of order $p~(\neq 2)$, then there exist a non-trivial element $y$ in $\left\langle x \right\rangle $. Then $x$ is not adjacent to $y$ in $\overline{\mathscr{P}^*(G)}$, which is a contradiction to the hypothesis. Therefore, $G$ is a 2-group with exponent 2. Since any group with exponent 2 must be abelian, so $G\cong \mathbb Z_{2}^m,~m\geq 1$. Conversely, if $G\cong \mathbb Z_{2}^m,~m\geq 1$, then every element of $G$ is of order 2, so it follows that $\overline{\mathscr{P}^*(G)}$ is complete.
\end{proof}
\begin{thm}
	Let $G$ be a finite group. Then $\overline{\mathscr{P}^*(G)}$ is $K_{1,3}$-free if and only if $G$ is isomorphic to one of the following:
	\begin{enumerate}[{\normalfont (i)}]
		\item $\mathbb Z_{p^n}$, $\mathbb Z_6$, $S_3$, $\mathbb Z_2^n$, $Q_8$, where $n\geq 1$;
		\item 3-group with exponent 3;
		\item non-nilpotent group of order $2^{n}.3$ or $2.3^m$, where $n,m>1$ with all non-trivial elements are of order $2$ or $3$.
		
	\end{enumerate}
\end{thm}
\begin{proof} Let $|G|$ has $k$ distinct prime divisors. \\
	\noindent\textbf{Case 1.} If $k\geq 3$, then $G$ contains at least one subgroup of order $p\geq 5$, let it be  $H$. Since $H$ is a subgroup of prime order, so the non-trivial  elements in $H$ are not adjacent to each other in $\overline{\mathscr{P}^*(G)}$. The elements in $G$ of order $q$ $(q\neq p)$, are not a power of any of the elements of $H$, and vise versa. So $\overline{\mathscr{P}^*(G)}$ contains $K_{1,3}$ as an induced subgraph.\\
	\noindent\textbf{Case 2.} If $k=2$, then $|G|={p}^{n}{q}^{m}$ where $n, m\geq 1$. If at least one of $p$ or $q\geq 5$, then $\overline{\mathscr{P}^*(G)}$ contains $K_{1,3}$ as an induced subgraph, as by the argument used in Case 1. So we now assume that both $p$, $q<5$; without loss of generality, we can take $p=2$, $q=3$.
	
	If 	$m=n=1$, then $G\cong \mathbb  Z_6$ or $S_3$. It is easy to see that $\overline{\mathscr{P}^*(\mathbb  Z_6)}\cong K_{1,2}\cup \overline{K_2}$, which is $K_{1,3}$-free.  $\overline{\mathscr{P}^*(S_3)}$ is shown in Figure~\ref{S3}, which is
	\begin{figure}[ht]
		\begin{center}
			\includegraphics[scale=1]{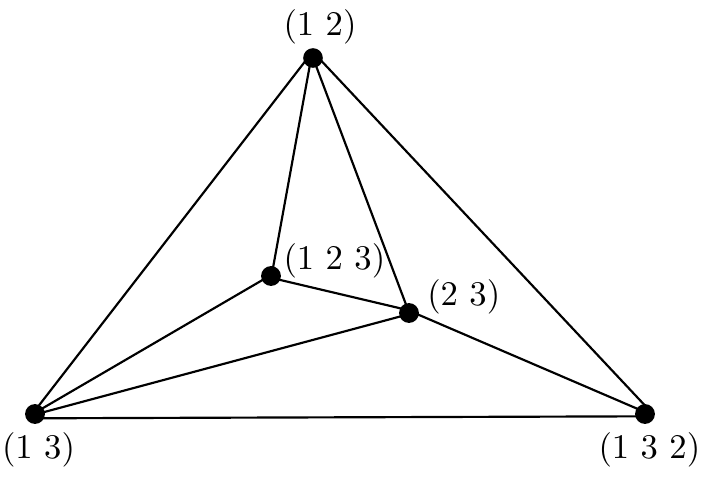}\caption{ $\overline{\mathscr{P}^*(S_3)}.$} \label{S3}
		\end{center}
	\end{figure}
	$K_{1,3}$-free.
	
	Now we assume that either $n$ or $m>1$. Suppose $G$ contains an element $x$ whose order is not a prime, then $o(x)=2^l~(l>1)$, $3^s~(s>1)$ or $2^l3^s$  $(0<l\leq n$, $0<s\leq m)$.
	
	\noindent\textbf{Subcase 2a.} If $o(x)=2^l$, $l>1$,
	then any three non-trivial elements of $\langle x \rangle$ are power of each other. These three elements together with the element of order 3 forms $K_{1,3}$ as an induced subgraph of $\overline{\mathscr{P}^*(G)}$.
	
	\noindent\textbf{Subcase 2b.} If $o(x)=3^s,s>1$, then as in Subcase 2a, we can show that $\overline{\mathscr{P}^*(G)}$ contains $K_{1,3}$ as an induced subgraph.
	
	\noindent\textbf{Subcase 2c.} If $o(x)=2^l3^s$, where $0<l\leq n$, $0<s\leq m$, then $\left\langle x\right\rangle $ contains an element of order 6, let it be $y$. Let $X_1$ and $X_2$ be subsets of $\left\langle y \right\rangle$, where $X_1$ contains two elements of order 6, and two elements of order 3; and $X_2$ contains one element of order 2, and two elements of order 6.	
	Since either $n$ or $m>1$, so $G$ contains a subgroup of order $p^s$, where $p=2$ or 3 and $s>1$, let it be $H$. Suppose that $H$ is cyclic. Then $G$ contains an element of order $p^s$, where $p=2$ or $3$ and $s>1$. By Subcases 2a and 2b, $\overline{\mathscr{P}^*(G)}$ contains $K_{1,3}$ as a subgraph.	
	If $H$ is non-cyclic, then $H$ contains more than two cyclic subgroups of order $p$. Hence $G$ contains an element of order $p$, which is not in $\left\langle y\right\rangle $, let it be $z$. Then the elements in $X_1$ and $z$ induces $K_{1,3}$ as a subgraph of $\overline{\mathscr{P}^*(G)}$. Thus,  it remains to consider the case when all the non-trivial elements of $G$  are of order either 2 or 3. If we assume that $G$ is such  a group, then by \cite{prime elt}, $G$ must be non-nilpotent of order either $2^n.3$ or $2.3^m$, $n,m>1$. Moreover, the degree of each vertex in $\overline{\mathscr{P}^*(G)}$ is either $|G|-2$ or $|G|-3$, and so $\overline{\mathscr{P}^*(G)}$ is  $K_{1,3}$-free.\\
	\noindent\textbf{Case 3.} If $k=1$, then $|G|=p^n$, $n\geq 1$.
	
	\noindent\textbf{Subcase 3a.} If G is cyclic, then obviously
	\begin{align}\label{p^n} \overline{\mathscr{P}^*(\mathbb Z_{p^n})}\cong \overline {K}_{p^{n}-1},
	\end{align}
	which is $K_{1,3}$-free.
	
	\noindent\textbf{Subcase 3b.} Assume that $G$ is non-cyclic.
	
	\noindent\textbf{Subcase 3b(i).} Let $p=2$. If $n=2$, then $G\cong \mathbb Z_2\times \mathbb Z_2$, and so $\overline{\mathscr{P}^*(G)}\cong K_3$, which is $K_{1,3}$-free. Now we assume that $n>2$.  If $G \cong \mathbb Z_2^n$, then by Theorem~\ref{complete}, $\overline{\mathscr{P}^*(G)} \cong K_{2^n-1}$, which is $K_{1,3}$-free.		If $G\cong Q_{8}$, then $\overline{\mathscr{P}^*(G)}$ is as shown in Figure~\ref{Q8}, which is $K_{1,3}$-free.
	\begin{figure}[ht]
		\begin{center}
			\includegraphics[scale=1]{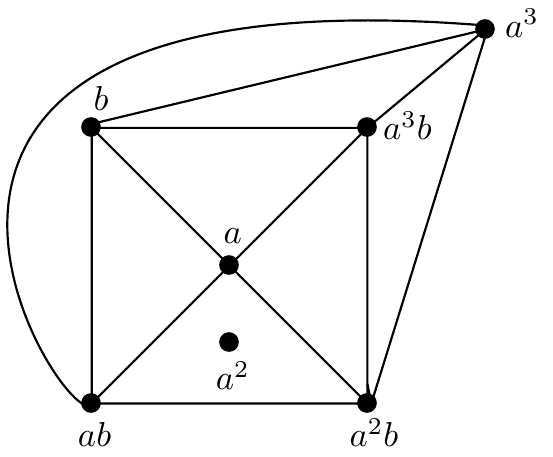}\caption{$\overline{\mathscr{P}^*(Q_8)}.$} \label{Q8}
		\end{center}
	\end{figure}
	If $G\cong Q_{2^n},$ $n\geq 4$, then $G$ contains a cyclic subgroup of order 8, let it be $H$. Here $H$ contains a unique subgroup of order 4. But $G$ contains at least two  cyclic subgroup of order 4, so as by the argument used in Case 1, $\overline{\mathscr{P}^*(G)}$ contains $K_{1,3}$ as an induced subgraph.
	
	Next, we assume that $G \ncong \mathbb Z_2^n$ and	 $Q_{2^n}$. Then $G$ contains an element of order $2^2$, let it be $x$ and so $\left\langle x \right\rangle $ contains a unique element of order 2. By Theorem~\ref{pre} (ii) and (iii), $G$  contains at least three elements of order 2. Therefore,  an element of order 2, which is not in $\left\langle x \right\rangle $ together with the non-trivial elements in $\left\langle x \right\rangle $ forms $K_{1,3}$ as an induced subgraph of $\overline{\mathscr{P}^*(G)}$.		
	
	\noindent\textbf{Subcase 3b(ii).} Let $p\neq 2$. Then by Theorem~\ref{pre}(i), $G$ has a subgroup $H \cong \mathbb Z_p\times \mathbb Z_p$. Then $H$ contains $p+1$ subgroups of order $p$. Also any two of these subgroups have trivial intersection. Hence each non-trivial element in any of these cyclic subgroups is not a power of any non-trivial element in another cyclic subgroups of $H$. Hence,
	\begin{align}\label{zpp}
	\overline{\mathscr{P}^*(\mathbb Z_{p}\times \mathbb Z_{p})}\cong K(p+1,p-1).
	\end{align}
	
	If $p\geq 5$, then by~(\ref{zpp}), $\overline{\mathscr{P}^*(G)}$ contains $K_{1,3}$ as an induced subgraph. Now assume that $p=3$. If $n=2$, then by~(\ref{zpp}), $\overline{\mathscr{P}^*(G)}\cong K(4,2)$. Therefore $\overline{\mathscr{P}^*(G)}$ is $K_{1,3}$-free. Suppose that $n>2$. If $G$ contains at least one element of order $3^2$, let it be $x$. Then $\left\langle x \right\rangle $ contains a unique subgroup of order 3. By Theorem~\ref{pre}(ii) and(iii), G contains at least four subgroups of order 3. Then as in the argument used in Case 1, the element $y\notin \left\langle x \right\rangle $ of order 3 together with the non-trivial elements in $\left\langle x \right\rangle $ forms $K_{1,3}$ as an induced subgraph of $\overline{\mathscr{P}^*(G)}$. If all the elements in $G$ are of order 3, that is, $G$ is a 3-group with exponent 3. Then degree of each vertex of $\overline{\mathscr{P}^*(G)}$ is $3^n-3$, so $\overline{\mathscr{P}^*(G)}$  is $K_{1,3}$-free.
	
	The proof follows by combining together all the above cases.
\end{proof}

\begin{thm}\label{c_3}
	Let $G$ be a finite group.  Then the following are equivalent:
	\begin{enumerate}[{\normalfont (1)}]
		\item $\overline{\mathscr{P}^*(G)}$ is isomorphic to either $\mathbb Z_{p^n}$ or $\mathbb Z_{pq^m}$, $n,m\geq 1$;
		\item $\overline{\mathscr{P}^*(G)}$ is $C_3$-free;
		\item $\overline{\mathscr{P}^*(G)}$ is bipartite.
	\end{enumerate}
\end{thm}
\begin{proof} First we prove that  $ (1)\Leftrightarrow$ (2): \\
	Let $|G|$ has $k$ distinct prime divisors. Now we divide the proof into the following cases. \\
	\noindent\textbf{Case 1.} If $k=1$, then $|G|=p^n$. Suppose $G$ is cyclic, then by (\ref{p^n}), $\overline{\mathscr{P}^*(G)}$ is totally disconnected. Now, we assume that $G$ is non-cyclic. If $p>2$, then by Theorem~\ref{pre}(i), $G$ contains a subgroup isomorphic to $\mathbb Z_p\times \mathbb Z_p$. Then by~(\ref{zpp}), $\overline{\mathscr{P}^*(G)}$ contains $C_3$ as a subgraph. Now, let $p=2$. If $G\ncong Q_n$, then by Theorem~\ref{pre}(ii) and (iii), $G$ contains at least 3 elements of order 2, and so $\overline{\mathscr{P}^*(G)}$ contains $C_3$ as a subgraph. If $G\cong Q_n$, then $G$ contains at least three cyclic subgroups of order 4, and so $\overline{\mathscr{P}^*(G)}$ contains $C_3$ as a subgraph. \\
	\noindent\textbf{Case 2.} If $k=2$, then $|G|=p^nq^m$. Suppose $G$ is cyclic, then the elements of order $p^n$, $q^m$ and $p^rq^s$, where $0<r<n$, $0<s<m$ are not powers of one another. So they form $C_3$ as a subgraph of $\overline{\mathscr{P}^*(G)}$. Suppose that either $n<2$ or $m<2$. Without loss of generality, we assume that $n=1$. Then every element of order $p$ is adjacent to the elements of order $q^s$, $0<s\leq m$; the elements of order $pq^s$ , $0<s\leq m$, are adjacent to the elements of order $q^t$, $t>s$. So $\overline{\mathscr{P}^*(G)}$ does not contains $C_3$ as a subgraph. Suppose $G$ is a non-cyclic abelian, then $G$ contains a  subgroup isomorphic to either $\mathbb Z_p\times \mathbb Z_p$ or $\mathbb Z_q\times \mathbb Z_q$, and so by~(\ref{zpp}), $\overline{\mathscr{P}^*(G)}$ contains $C_3$ as  a subgraph. Suppose that $G$ is non-abelian. If $n=m=1$ and $q>p$, then $G\cong \mathbb Z_q\rtimes \mathbb Z_p$ and it contains $q$ Sylow $p$-subgroups, and a unique  Sylow $q$-subgroup. So
	\begin{align}\label{zqzp}
	\overline{\mathscr{P}^*(\mathbb Z_q\rtimes \mathbb Z_p)}\cong K(q,p-1)+\overline{K}_{q-1},
	\end{align}
	which contains $C_3$ as a subgraph. If either $n>1$ or $m>1$, then $G$ contains a subgroups of order $p^n$ and $q^m$, let them be $H$ and $K $ respectively. If either $H$ or $K$ is non-cyclic, then by Case 1, $\overline{\mathscr{P}^*(G)}$ contains $C_3$. If $H$ and $K$ are cyclic, then $G$ contains elements of order $p^r$, $0<r\leq n$ and $q^s$, $0<s\leq m$. Let $z$ be an element in $G$, which is not in $H$ and $K$. If the order of $z$ is $p^r$, where $r\leq n$, then $z$ together with an element of order $p^n$ and the element of order $q$ forms $C_3$ in $\overline{\mathscr{P}^*(G)}$. Similarly, if the order of $z$ is $q^s$, where $s\leq m$, then $\overline{\mathscr{P}^*(G)}$ contains $C_3$. If order of $z$ is $p^rq^s$, then $G$ contains an element of order $pq$. This element together with an element of order $p^n$ and $q^m$ forms $C_3$ in $\overline{\mathscr{P}^*(G)}$. \\
	\noindent\textbf{Case 3.}  If $k\geq 3$, then $G$ contains at least three elements of distinct prime orders, and  so they forms $C_3$ in $\overline{\mathscr{P}^*(G)}$.
	
	Next, we show that (1) $\Rightarrow$ (3): If $G\cong \mathbb Z_{p^n}$, then by (\ref{zpp}), $\overline{\mathscr{P}^*(G)}$ is bipartite.
	If $G\cong \mathbb Z_{pq^m}$, $m\geq 1$, then $\overline{\mathscr{P}^*(G)}$ is bipartite with bipartition $X$ and $Y$, where $X$ contains the elements of order $q^s$, $0<s\leq m$, and $Y$ contains the elements of order $p$ and the elements of order $pq^s$, $0<s\leq m$. \\
	Proof of (3) $\Rightarrow$ (2) is obvious.
	
	Combining all the above cases, we get the result.
\end{proof}
\begin{thm}\label{connected}
	Let $G$ be a finite group. Then $\overline{\mathscr{P}^*(G)}$ is disconnected if and only if $G\cong \mathbb Z_n$ or $Q_{2^\alpha}$. In this case, the number of  components of $\overline{\mathscr{P}^*(\mathbb Z_n)}$ is $n-1$, if $n=p^{\alpha}$; $\varphi{(n)}+1$, otherwise. The number of components of $\overline{\mathscr{P}^*(Q_{2^\alpha})}$ is 2.
\end{thm}

\begin{proof}
	First we assume that $G\cong \mathbb Z_n$. Then all the elements of $G$ are powers of the generators of $G$. So the generators of $G$ are isolated vertices in $\overline{\mathscr{P}^*(G)}$. If $n=p^{\alpha}$, then by (\ref{p^n}), number of components of $\overline{\mathscr{P}^*(G)}$ is $n-1$. Assume that $n=p_1^{n_1}p_2^{n_2}\ldots p_k^{n_k}$, where $p_i$'s are distinct primes and $n_i \geq 1$ for all $i$ and $k>1$. Let $x$ and $y$ be non-generators  of $G$ such that they are non-adjacent in $\overline{\mathscr{P}^*(G)}$. Then $\left\langle x \right\rangle \subseteq \left\langle y \right\rangle $ or $\left\langle y \right\rangle \subseteq \left\langle x \right\rangle $. Without loss of generality, we assume that $\left\langle y \right\rangle \subseteq \left\langle x \right\rangle $. Since $x$ is a non-generators  of $G$, so $p_i^{n_i}\nmid o(x)$ for some $i$. Hence $x$ adjacent to the element of order $p_i^{n_i}$ in $\overline{\mathscr{P}^*(G)}$, say $z$. Then $y$ is also adjacent to $z$, so $x-z-y$ is a $x-y$ path in $\overline{\mathscr{P}^*(G)}$. Thus, $\overline{\mathscr{P}^*(G)}$ has $\varphi{(n)}+1$ components.
	
	If $G\cong Q_{2^\alpha}$, then $G$ contains one cyclic subgroup of order $2^{\alpha-1}$, say $H$ and the remaining elements are of order 4. Since the element of order 2, say $x$ is a power of all the elements of $G$. So this element is an isolated vertex in $\overline{\mathscr{P}^*(G)}$. Hence all the elements in $G$ of order $4$, which is not in $H$ are adjacent to all the non-trivial elements of $H$, except $x$. So the number of components of $\overline{\mathscr{P}^*(G)}$ is 2.
	
	Now we assume that $G\ncong \mathbb Z_n$ and $Q_{2^\alpha}$. We have to show that $\overline{\mathscr{P}^*(G)}$ is connected. Let $|G|=p_1^{n_1}p_2^{n_2}\cdots p_k^{n_k}$, where $p_i$'s are distinct primes. We need to consider the following two cases.\\
	\noindent \textbf{Case 1.} Let $k \geq 2$.  For each $i=1,2,\ldots , k$, Let $X_{i}=\left\lbrace x\in G~|~o(x)=p_i^{m_i}, 0<m_i\leq n_i \right\rbrace $. Then each element $x \in X_i$ is adjacent to all the elements in $X_j,$ $i\neq j$. So the subgraph induced by the elements of $\displaystyle \bigcup_{i=1}^{k}X_i$ is connected. Now let $x\in G$ with $x\notin \displaystyle \bigcup_{i=1}^{k}X_i$. Since $G$ is non-cyclic, so $p_i^{n_i}\nmid o(x)$ for some $i$. Let $H$ be Sylow $p_i$-subgroup of $G$.\\
	\noindent\textbf{Subcase 1a.} If $H$ is cyclic, then $G$ contains an element of order $p_i^{n_i}$, say $z$. Then $z$ is not a power of $x$. So $z\in \bigcup_{i=1}^{k}X_i$, and is adjacent to $x$ in $\overline{\mathscr{P}^*(G)}$. \\
	\noindent\textbf{Subcase 1b.} Let $H$ be non-cyclic. If $H$ is non-quarternion, then by Theorem~\ref{pre} (ii) and (iii), $G$ contains more than two cyclic subgroups of order $p$.  So $x$ is adjacent to the elements of order $p$, which are not in $\left\langle x \right\rangle $. If $H$ is quarternion, then $H$ contains more than two cyclic subgroup of order 4, so $x$ is adjacent to the elements of order $4$, which are not in $\left\langle x \right\rangle $. In both the cases there exist $z\in \displaystyle \bigcup_{i=1}^{k}X_i$, which adjacent to $x$ in $\overline{\mathscr{P}^*(G)}$. \\
	\noindent \textbf{Case 2.} Let $k=1$.  Since $G$ is non-cyclic and non-quaternion, so by Theorem~\ref{pre}(ii) and (iii), $G$ contains more than two cyclic subgroups of order $p_1$, let them be $H_i:= \langle z_i \rangle$, $i=1,2,\ldots,r$, for some $r\geq 3$. Then each non-trivial element in $H_i$ is not a power of any non-trivial element in $H_j,$ $i\neq j$. Hence the subgraph induced by an elements in $\displaystyle \bigcup_{i=1}^{k}H_i$ is connected. Now, let $x\in G$, with $x\notin \displaystyle \bigcup_{i=1}^{k}H_i$. Then $\left\langle x \right\rangle $ contains exactly one $z_i$, so any $z_j$, $j\neq i$ is not a power of $x$; since $o(x)>o(z_j)$, so $x$ is also not a power of $z_j$. Thus $x$ is adjacent to $z_j\in \displaystyle \bigcup_{i=1}^{k}H_i$.
	
	From the above arguments, it follows that $\overline{\mathscr{P}^*(G)}$ is connected. This completes the proof.
\end{proof}

From the proof of the previous theorem, we deduce the following result. Note that this result follows directly from \cite[Lemma 8]{Mogha}. Here we obtain this as a consequence of the previous theorem.
\begin{cor}\label{isolated}
	Let $G$ be a finite group. Then $\overline{\mathscr{P}^*(G)}$ contains isolated vertices if and only if $G\cong \mathbb Z_n$ or $Q_{2^m}$. Moreover, the number of isolated vertices in $\overline{\mathscr{P}^*(\mathbb Z_n)}$ is $n-1$, if $n=p^{\alpha}$; $\varphi{(n)}$, otherwise. The number of isolated vertices in $\overline{\mathscr{P}^*(Q_{2^m})}$ is 1.
\end{cor}
%

\begin{thm}
	Let $G$ be a finite group. Then diam$(\overline{\mathscr{P}^*(G)})$ is $\infty$, if $~G\cong \mathbb Z_n$  or $Q_n$; 1, if $~G\cong \mathbb Z_2^m,~m\geq 1$; 2,  otherwise.
\end{thm}
\begin{proof}
	The possibilities of $G$ with diam$(\mathscr{P}^*(G))$ is either $\infty$ or 1 follows from  Theorems~\ref{connected} and \ref{complete}, respectively. Now we assume that $G\ncong \mathbb Z_n$, $Q_n$ or $\mathbb Z_2^m,m\geq 1$. Let $G$ has $k$ distinct prime factors.\\
	\noindent \textbf{Case 1.} Let $k=1$. Since $G\ncong \mathbb Z_n$ and $Q_{2^\alpha}$, so by parts (ii), (iii) of Theorem~\ref{common}, $G$ contains at least three subgroups of prime order, let them be $\left\langle x_i \right\rangle ,$ $i=1,2,\ldots,r$, where $r\geq 3$. Now, let $x$ be a non-trivial element in $G$. Then $\left\langle x \right\rangle $ contains exactly one $\left\langle x_i \right\rangle $, for some $i$. It follows that,  every $x_j$ $(j\neq i)$ is not a power of $x$, and vice versa, so $x$ is adjacent to all $x_j$ $(j\neq i)$. Now let $u,~v$ be non-trivial elements in $G$. Then $\left\langle u \right\rangle $ and $\left\langle v \right\rangle $ contains a subgroup of prime order, let them be $\left\langle x_r\right\rangle $ and $\left\langle x_s\right\rangle $, respectively. If $\left\langle x_r\right\rangle \neq \left\langle x_s\right\rangle $, then $u$ and $v$ are not a powers of each other. So $u$ and $v$ are adjacent in $\overline{\mathscr{P}^*(G)}$. If $\left\langle x_r\right\rangle = \left\langle x_s\right\rangle$, then there exist $x_l$ ($ l \neq i$), which  is adjacent to both $u$ and $v$ in $\overline{\mathscr{P}^*(G)}$. So $u-x_l-v$ is a  $u-v$ path in $\overline{\mathscr{P}^*(G)}$.
	
	\noindent \textbf{Case 2.} Let $k\geq 2$. Let  $|G|=p_1^{n_1}p_2^{n_2}\ldots p_k^{n_k}$, where $p_i$'s are distinct primes and $n_i \geq 1$ for all $i$. Let $x$ and $y$ be not adjacent vertices in $\overline{\mathscr{P}^*(G)}$. Then $\left\langle x \right\rangle \subseteq \left\langle y \right\rangle $ or $\left\langle y \right\rangle \subseteq \left\langle x \right\rangle $. Without loss of generality, we assume that $\left\langle y \right\rangle \subseteq \left\langle x \right\rangle $. Since $G$ is non-cyclic, so $p_i^{n_i}\nmid o(x)$ for some $i$. Then by Subcases 1a and 1b in proof of Theorem~\ref{connected}, $x$ is adjacent to the element of order $p_i^{n_i}$ or $p_i$, let that element be $z$. Then $y$ is also adjacent to $z$, so $x-z-y$ is $x-y$ path in $\overline{\mathscr{P}^*(G)}$
\end{proof}
\begin{thm}
	If $G$ is a finite group, then gr$(\overline{\mathscr{P}^*(G)})$ is $\infty $, if $G\cong\mathbb Z_{p^n} $ or $\mathbb Z_{2p}$; 4, if $G\cong \mathbb Z_{pq^m}$, where $q^m\neq 2$; 3, otherwise.
\end{thm}
\begin{proof}
	If $G\ncong \mathbb Z_{p^n},$ and $\mathbb Z_{pq^m},$ $n,m\geq 1$, then by Theorem~\ref{c_3}, $G$ contains $C_3$ as a subgraph. If $G\cong \mathbb Z_{p^n}$, then by (\ref{p^n}),  $\overline{\mathscr{P}^*(\mathbb Z_{p^n})}$ is acyclic. Now let $ G \cong \mathbb Z_{pq^m}$. If  $m\geq 1$, then  by Theorem~\ref{c_3}, $\overline{\mathscr{P}^*(\mathbb Z_{pq^m})}$ is bipartite;
	If $m>1$, then $G$ contains $C_4$ as a subgraph of $\overline{\mathscr{P}^*(G)}$; If $m=1$, then the non-trivial elements of $G$ are of orders one of $p,q,pq$. The elements of order $pq$ are generators of $G$, and hence they are isolated vertices in $\overline{\mathscr{P}^*(G)}$. Also the elements of order $p$ and $q$ are not a power of one another. Therefore,
	\begin{align}\label{pq}
	\overline{\mathscr{P}^*(\mathbb Z_{pq})} \cong K_{p-1,q-1}\cup \overline{K}_{(p-1)(q-1)},
	\end{align} which is acyclic, when $p=2$, and it contains $C_4$, when $p >2$. The proof follows from these facts.
\end{proof}

\section{ Embedding of the complement of proper power graphs of groups on topological surfaces}

The main results we prove in this section are the following:

\begin{thm}\label{planar}
	Let $G$ be a finite group and $p$ be prime. Then
	\begin{enumerate}[{\normalfont (1)}]
		\item  $\overline{\mathscr{P}^*(G)}$ is planar if and only if $G$ is one of  $\mathbb Z_{p^{\alpha}}, ~\mathbb Z_{12}, ~\mathbb Z_{2p}, ~\mathbb Z_{3p}$, $\mathbb Z_{2}\times \mathbb Z_{2}$, $Q_8$, $S_3$;
		\item $\overline{\mathscr{P}^*(G)}$ is toroidal if and only if $G$ is  one of  $\mathbb Z_{18}$, $\mathbb Z_{20}$, $\mathbb Z_{28}$, $\mathbb Z_{3}\times \mathbb Z_{3}$, $\mathbb Z_{2}\times \mathbb  Z_{2} \times \mathbb Z_{2}$,  $\mathbb Z_{4}\times \mathbb Z_{2}$, $D_8$;
		\item $\overline{\mathscr{P}^*(G)}$ is projective if and only if $G$ is  one of  $\mathbb Z_{20}$, $\mathbb Z_{4}\times \mathbb Z_{2}$, $D_8$.
	\end{enumerate}
\end{thm}

As a consequence of this result, we deduce the following:
\begin{cor}\label{cor1}
	Let $G$ be a finite group. Then
	\begin{enumerate} [{\normalfont (1)}]
		\item $\overline{\mathscr{P}^*(G)}$ is neither a path nor a star;
		\item $\overline{\mathscr{P}^*(G)}$ is $C_n$ if and only if $n=3$ and $G\cong \mathbb Z_2\times \mathbb Z_2$.
		\item $\overline{\mathscr{P}^*(G)}$ does not contain $K_{1,4}$ as a subgraph if and only if $G$ is either $\mathbb Z_{p^n}$, $\mathbb Z_6$ or $\mathbb Z_2\times \mathbb Z_2$;
		\item The following are equivalent:
		\begin{enumerate}[{\normalfont (a)}]
			\item $G$ is one of  $\mathbb Z_{p^n}, \mathbb Z_2\times \mathbb Z_2$ or $\mathbb Z_{2p}$;
			\item  $\overline{\mathscr{P}^*(G)}$ is outerplanar;
			\item  $\overline{\mathscr{P}^*(G)}$ does not contain $K_{2,3}$ as a subgraph.
		\end{enumerate}
	\end{enumerate}
\end{cor}

First, we begin with the following result.

\begin{pro}\label{common}
	If $G$ is a finite group whose order has more than two distinct prime divisors, then $\gamma ({\overline{\mathscr{P}^*(G)}})>1$ and $\overline{\gamma} {(\overline{\mathscr{P}^*(G)})}>1$.
\end{pro}
\begin{proof} Let $|G|=p_1^{n_1}p_2^{n_2}\cdots p_k^{n_k}$, where $p_i$'s are distinct primes, $n_i \geq 1$ and $k\geq 3$.
	We divide the proof in to the following cases:\\
	\noindent\textbf{Case 1.} If $k=3$, then without loss of generality, we assume that  $p_1<p_2<p_3$. Let us consider the following subcases. \\
	\noindent\textbf{Subcase 1a.} If $p_1>2$, then $G$ contains at least two elements of order $p_1$, at least four elements of order $p_2$, and at least six elements of order $p_3$. Then the elements of order $p_1^{\alpha_1}$ $(0< {\alpha}_1\leq n_1)$ and $p_2^{\alpha_2}$ $(0< {\alpha}_2\leq n_2)$ are adjacent to the elements of order $p_3^{\alpha_3}$ $(0< {\alpha}_3\leq n_3)$ in $\overline{\mathscr{P}^*(G)}$. Hence $\overline{\mathscr{P}^*(G)}$ contains $K_{5,6}$ as a subgraph, and so $\gamma ({\overline{\mathscr{P}^*(G)}})>1$ and $\overline{\gamma} {(\overline{\mathscr{P}^*(G)})}>1$.\\
	\noindent\textbf{Subcase 1b.} Let $p_1=2$. If $p_2>3$, then $G$ contains at least four elements of order $p_2$, and at least six elements of order $p_3$. Then the elements of order 2 and $p_2$ are adjacent to the elements of order $p_3$. Thus $\overline{\mathscr{P}^*(G)}$ contains $K_{5,4}$ as a subgraph, and so $\gamma ({\overline{\mathscr{P}^*(G)}})>1$ and $\overline{\gamma} {(\overline{\mathscr{P}^*(G)})}>1$.
	
	If $p_2=3$, and either $n_2 \geq 2$ or $n_3 \geq 2$, then $\overline{\mathscr{P}^*(G)}$ contains $K_{3,7}$ as a subgraph, so $\gamma ({\overline{\mathscr{P}^*(G)}})>1$ and $\overline{\gamma} {(\overline{\mathscr{P}^*(G)})}>1$. Now, we assume that $n_2=n_3=1$. Suppose for some $i~(i=1,2,3)$,  the Sylow $p_i$-subgroup is not unique. If $i=1$, then $G$ contains at least three elements of order 2. Then the elements of order 2 and $3$ are adjacent to the elements of order $p_3$. Hence $\overline{\mathscr{P}^*(G)}$ contains $K_{5,4}$ as a subgraph. If $i=2$ or $3$, then $G$ contains at least 8 elements of order $p_i$, so $\overline{\mathscr{P}^*(G)}$ contains $K_{3,7}$ as a subgroup.  Suppose for each $i$, Sylow $p_i$-
	subgroup of $G$ is unique, then $G\cong P\times \mathbb Z_{3}.{p_3}$, where $P$ is the
	Sylow $2$-subgroup of order $2^{n_1}$. If $n_1=1$, then $G\cong \mathbb Z_{6}.{p_3}$. In this case, $\overline{\mathscr{P}^*(G)}$  contains $K_{5,4}$ as a subgraph. If $n_1\geq 1$, then $\overline{\mathscr{P}^*(G)}$ contains $K_{3,7}$ as a subgraph. In both the cases, we have  $\gamma ({\overline{\mathscr{P}^*(G)}})>1$ and $\overline{\gamma} {(\overline{\mathscr{P}^*(G)})}>1$. \\
	\noindent\textbf{Case 2.} Let $k\geq 4$.  Let $p_i,$ $p_j,$ $p_r>2$, for some $i,j, r$. Then the elements of order $p_i$ and $p_j$ are adjacent to the elements of order $p_r$ in $\overline{\mathscr{P}^*(G)}$. Thus $\overline{\mathscr{P}^*(G)}$ contains $K_{3,7}$ as a subgraph, and so $\gamma ({\overline{\mathscr{P}^*(G)}})>1$ and $\overline{\gamma} {(\overline{\mathscr{P}^*(G)})}>1$. 			
	
	Proof follows by combining the above cases together.
\end{proof}

Proposition~\ref{common} reveals that, to prove the main result, it is enough to deal with the groups whose order has at most two distinct prime divisors. For this purpose, first we consider the finite cyclic groups, then we deal with the finite non-cyclic groups.
\begin{pro}\label{cyclic}
	Let $G$ be a finite cyclic group and $p$ be a prime. Then
	\begin{enumerate}[{\normalfont (1)}]
		\item $\overline{\mathscr{P}^*(G)}$ is planar if and only if $G$ is  one of $\mathbb Z_{p^{\alpha}}, ~\mathbb Z_{12}, ~\mathbb Z_{2p}, ~\mathbb Z_{3p}$;
		\item $\overline{\mathscr{P}^*(G)}$ is toroidal if and only if $G$ is one of  $\mathbb Z_{18}$, ~$\mathbb Z_{20}$, ~$\mathbb Z_{28}$;
		\item $\overline{\mathscr{P}^*(G)}$ is projective if and only if $G \cong \mathbb Z_{20}$.
	\end{enumerate}
\end{pro}
\begin{proof}
	Let $|G|$ has $k$ distinct prime divisors. Now we divide the proof into the following cases. \\
	\noindent\textbf{Case 1.} If $k=1$, then by~(\ref{p^n}), $\overline{\mathscr{P}^*(G)}$ is planar.\\
	\noindent\textbf{Case 2.} Let $k=2$. Let $H$ and $K$ be subgroups of $\mathbb Z_{p^nq^m}$ of order $p^n$ and $q^m$, respectively, where $n,m\geq 1$. The order of each non-trivial element in $H$ is relatively prime to the non-trivial elements in $K$. So no element in $H$ is not a power of any element in $K$ and vice versa. Therefore, $\overline{\mathscr{P}^*(G)}$ contains $K_{p^n-1,q^m-1}$ as a subgraph.
	If $n,m\geq 2$, then $\overline{\mathscr{P}^*(G)}$ contains $K_{3,7}$ as a subgraph, so  $\gamma ({\overline{\mathscr{P}^*(G)}})>1$ and $\overline{\gamma} {(\overline{\mathscr{P}^*(G)})}>1$. Now, we assume that either $n=1$ or $m=1$. Without loss of generality, we assume that $m=1$. Then $|G|=p^nq$, $n\geq 1$. We need to consider the following subcases: \\
	\noindent\textbf{Subcase 2a.} If $n=1$, then $G\cong \mathbb Z_{pq}$. By~(\ref{pq}), $\gamma ({\overline{\mathscr{P}^*(G)}})>1$ and $\overline{\gamma} {(\overline{\mathscr{P}^*(G)})}>1$ if $p,q\geq 5$; otherwise, $\overline{\mathscr{P}^*(G)}$ is planar.	 \\
	\noindent\textbf{Subcase 2b.} Let $n\geq 2$.\\
	\noindent\textbf{Subcase 2b(i).} 	
	Let $p=2$. If $n>2$, then $G$ contains four elements of order $8$, which are not a power of any elements of orders $q$, $2q$, $4q$ and vice versa. Hence $\overline{\mathscr{P}^*(G)}$ contains $K_{4,5}$ as a subgraph.
	Therefore, $\gamma ({\overline{\mathscr{P}^*(G)}})>1$ and $\overline{\gamma} {(\overline{\mathscr{P}^*(G)})}>1$.
	
	\begin{figure}[ht]
		\begin{center}
			\includegraphics[scale=1]{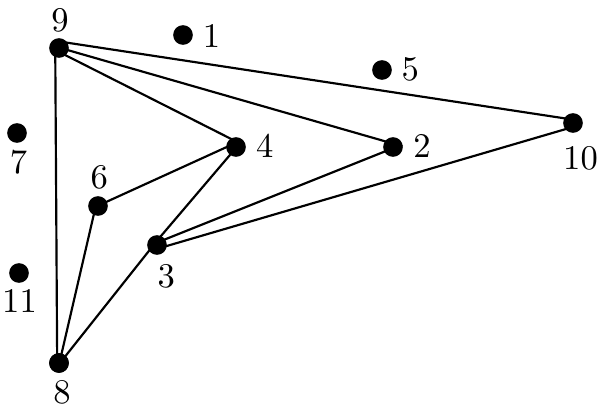}\caption{A plane embedding of $\overline{\mathscr{P}^*(\mathbb Z_{12})}.$} \label{Z12}
		\end{center}
	\end{figure}
	
	\begin{figure}[ht]
		\begin{center}
			\includegraphics[scale=1]{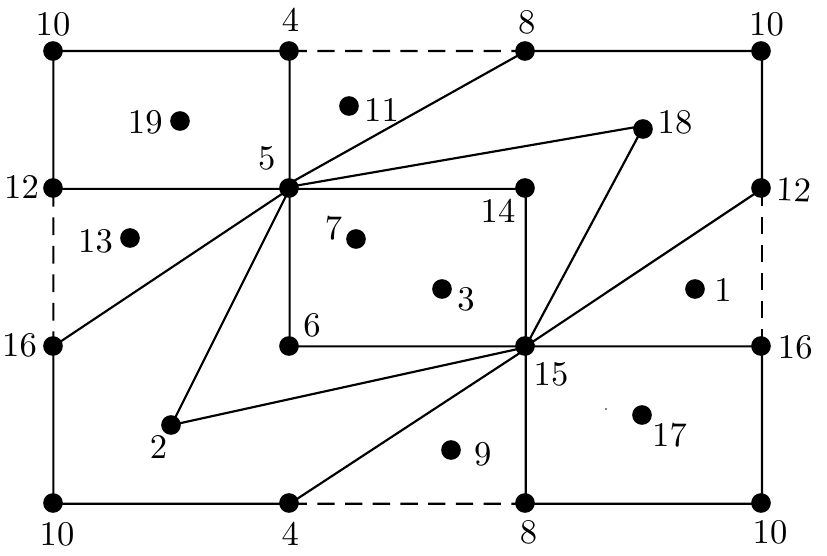}\caption{A toroidal embedding of $\overline{\mathscr{P}^*(\mathbb Z_{20})}.$} \label{Z_20}.
		\end{center}
	\end{figure}
	
	If $n=2$, then in $G$, the elements of order 2, 4 are not a power of elements of order $q$ and vice versa. Also the element of order $2$ is a power of the elements of order $2q$; the elements of order $2q$ and 4 are not a power of each other. It follows that if $q=3$, then $\overline{\mathscr{P}^*(G)}$ is planar, and a plane embedding of $\overline{\mathscr{P}^*(G)}$ is shown in Figure~\ref{Z12}.

	\begin{figure}[ht]
		\begin{center}
			\includegraphics[scale=1]{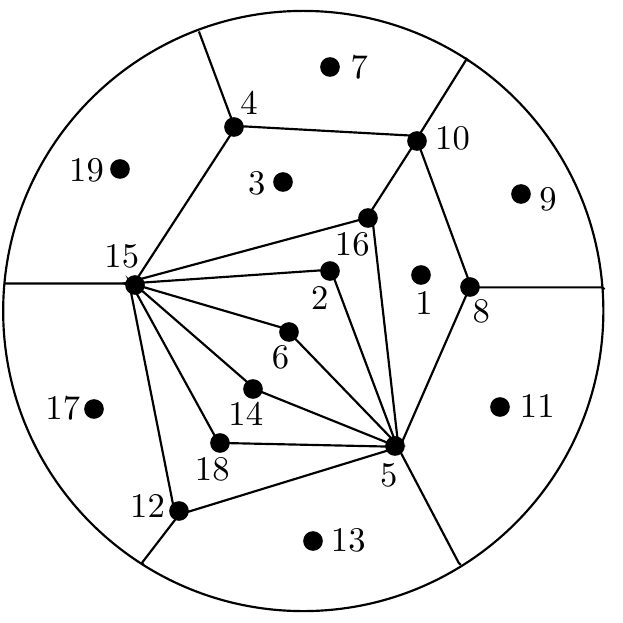}\caption{A projective embedding of $\overline{\mathscr{P}^*(\mathbb Z_{20})}.$} \label{Z_20P}.
		\end{center}
	\end{figure}	
	\begin{figure}[ht]
		\begin{center}
			\includegraphics[scale=1]{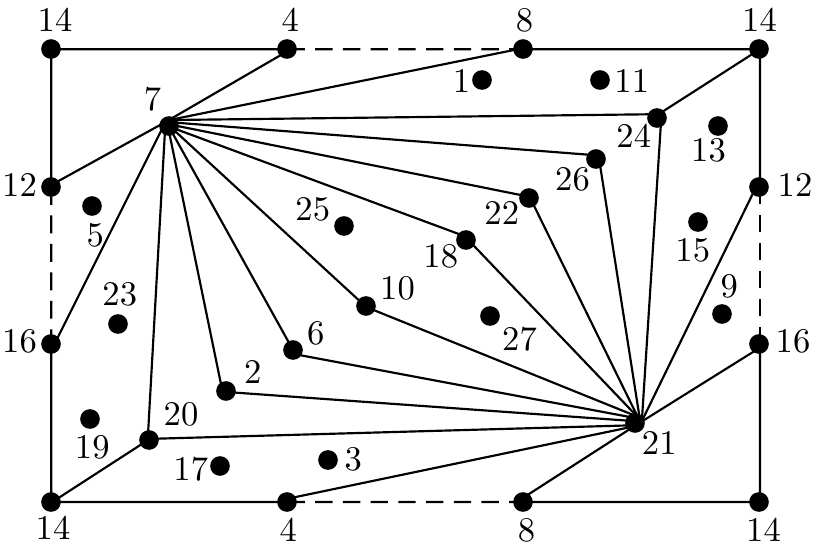}\caption{A toroidal embedding of $\overline{\mathscr{P}^*(\mathbb Z_{28})}.$} \label{Z_28}
		\end{center}
	\end{figure}
	If $q=5$, then $\gamma ({\overline{\mathscr{P}^*(G)}})=1$, $\overline{\gamma} {(\overline{\mathscr{P}^*(G)})}=1$, and a toroidal and a projective embeddings of $\overline{\mathscr{P}^*(G)}$ are shown in Figures~\ref{Z_20} and ~\ref{Z_20P}, respectively.
	If $q=7$, then $\overline{\mathscr{P}^*(G)}$ contains $K_{3,6}$ as a subgraph, so $\overline{\gamma} {(\overline{\mathscr{P}^*(G)})}>1$, $\gamma( {\overline{\mathscr{P}^*(G)})}=1$	
	and a toroidal embedding of $\overline{\mathscr{P}^*(G)}$ is shown in Figure~\ref{Z_28}. If $q>7$, then $\overline{\mathscr{P}^*(G)}$ contains $K_{3,7}$ as a subgraph, so $\overline{\gamma} {(\overline{\mathscr{P}^*(G)})}>1$ and $\gamma {(\overline{\mathscr{P}^*(G)})}>1$.
	
	\noindent\textbf{Subcase 2b(ii).}  Let $p=3$. If $n>2$, then $G$ contains eighteen elements of order $27$, which are not a power of any elements of orders $q,$ $3q,$ $9q$, and vice versa. Hence $\overline{\mathscr{P}^*(G)}$ contains $K_{4,5}$ as a subgraph. Therefore, $\gamma ({\overline{\mathscr{P}^*(G)}})>1$ and $\overline{\gamma} {(\overline{\mathscr{P}^*(G)})}>1$. If $n=2$, then in $G$, the elements of order 3, 9 are not a power of elements of order $q$, and vice versa. Also the element of order $3$ is a power of the elements of order $3q$; the elements of order $3q$ and 9 are not a power of each other. It follows that if $q=2$, then $\overline{\mathscr{P}^*(G)}$ contains $K_{3,6}$ as a subgraph, so $\overline{\gamma} {(\overline{\mathscr{P}^*(G)})}>1$; but $\gamma ({\overline{\mathscr{P}^*(G)}})=1$, and a toroidal embedding of $\overline{\mathscr{P}^*(G)}$ is shown in Figure \ref{Z18}.
	\begin{figure}[ht]
		\begin{center}
			\includegraphics[scale=1]{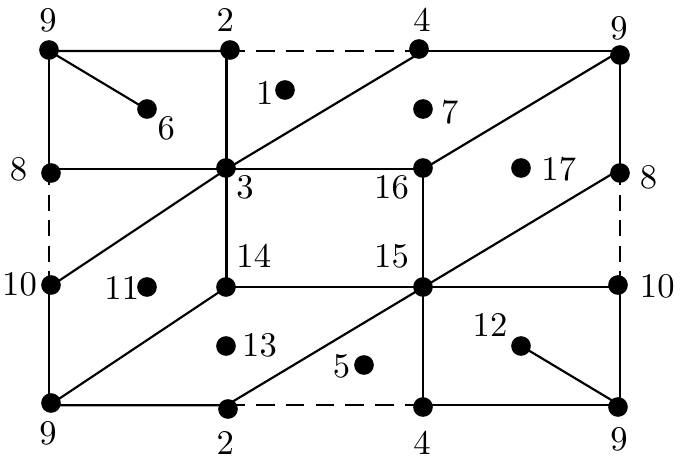}\caption{A toroidal embedding of $\overline{\mathscr{P}^*(\mathbb Z_{18})}.$} \label{Z18}
		\end{center}
	\end{figure}
	If $q\geq 5$, then $\overline{\mathscr{P}^*(G)}$ contains $K_{5,4}$ as a subgraph, so $\gamma ({\overline{\mathscr{P}^*(G)}})>1$ and $\overline{\gamma} {(\overline{\mathscr{P}^*(G)})}>1$. \\
	\noindent\textbf{Subcase 2b(iii)}: If $p\geq 5$, then the elements in $G$ of order $q$ and $pq$ are not a power of the elements of order $p^2$, and vise versa. Note that $G$ contains at least one element of orders $q$, at least four elements of order $pq$, and at least twenty elements of order $p^2$. It follows that $\overline{\mathscr{P}^*(G)}$ contains $K_{3,7}$ as a subgraph, so $\gamma ({\overline{\mathscr{P}^*(G)}})>1$ and $\overline{\gamma} {(\overline{\mathscr{P}^*(G)})}>1$. \\
	\noindent\textbf{Case 3.} Let $k\geq 3$. Then by Proposition~\ref{common}, $\gamma ({\overline{\mathscr{P}^*(G)}})>1$ and $\overline{\gamma} {(\overline{\mathscr{P}^*(G)})}>1$. \\
	Combining all the above cases  together, the proof follows.
\end{proof}
\begin{pro}\label{p^m}
	Let $G$ be a finite non-cyclic group of order ${p}^{\alpha}$, where $p$ is a prime and $\alpha \geq 2$. Then
	\begin{enumerate}[{\normalfont (1)}]
		\item $\overline{\mathscr{P}^*(G)}$ is planar if and only if $G$ is  either $\mathbb Z_{2}\times \mathbb Z_{2}$ or $Q_8$;
		\item $\overline{\mathscr{P}^*(G)}$ is toroidal if and only if $G$ is  one of $\mathbb Z_{3}\times \mathbb Z_{3}$, $\mathbb Z_{2}\times \mathbb  Z_{2} \times \mathbb Z_{2}$,  $\mathbb Z_{8}\times \mathbb Z_{2}$ or $D_8$;
		\item $\overline{\mathscr{P}^*(G)}$ is projective if and only if $G$ is  one of $\mathbb Z_{4}\times \mathbb Z_{2}$ or $D_8$.
	\end{enumerate}
\end{pro}
\begin{proof}
	We divide the proof into several cases.\\
	\noindent\textbf{Case 1.} Let $\alpha =2$. Then  $ G \cong\mathbb Z_{p}\times \mathbb Z_{p}$.
	If $p\geq 5$, then by~(\ref{zpp}), $\overline{\mathscr{P}^*(G)}$ contains two copies of $K_{3,3}$, and so $\gamma ({\overline{\mathscr{P}^*(G)}})>1$ and $\overline{\gamma} {(\overline{\mathscr{P}^*(G)})}>1$. If $p=3$, then by~(\ref{zpp}), $\overline{\mathscr{P}^*(G)}\cong K_{2,2,2,2}$. Here $\overline{\mathscr{P}^*(G)}$ contains $K_{4,4}$ as a subgraph, so $\overline{\gamma} {(\overline{\mathscr{P}^*(G)})}>1$; but $\gamma ({\overline{\mathscr{P}^*(G)}})=1$. A toroidal embedding of $\overline{\mathscr{P}^*(G)}$ is shown in Figure~\ref{Z3Z3}.
	If $p=2$, then
	\begin{align}\label{z22}
	\overline{\mathscr{P}^*(G)} \cong C_3,
	\end{align} which is planar.
	\begin{figure}[ht]
		\begin{center}
			\includegraphics[scale=1]{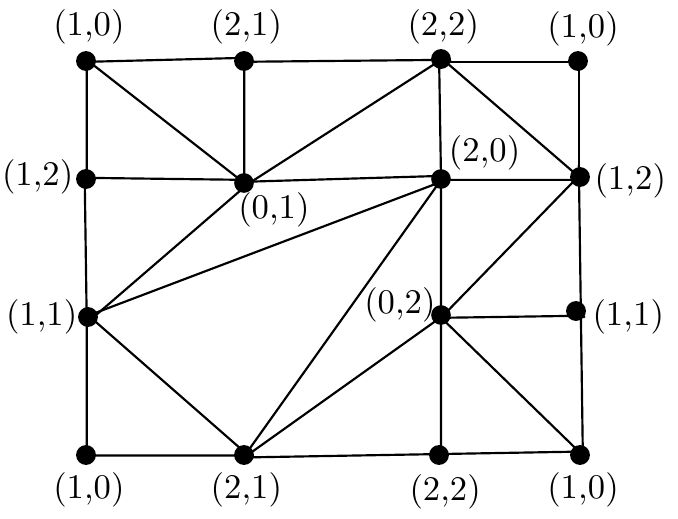}\caption{A toroidal embedding of $\overline{\mathscr{P}^*(\mathbb Z_{3}\times \mathbb Z_3)}.$} \label{Z3Z3}.
		\end{center}
	\end{figure}
	
	\noindent\textbf{Case 2.} Let $\alpha =3$.
	
	\noindent\textbf{Subcase 2a.} Assume that $p\geq 3$. Then up to isomorphism the only non-cyclic groups of order $p^3$ are $\mathbb Z_p \times \mathbb Z_p \times \mathbb Z_p$, $\mathbb Z_{p^2} \times \mathbb Z_p$, $(\mathbb Z_p\times \mathbb Z_p)\rtimes \mathbb Z_p$ and $M_{{p}^3}$.
	
	\noindent\textbf{Subcase 2a(i).} If $G\cong \mathbb Z_p \times \mathbb Z_p \times \mathbb Z_p$, then $G$ contains $p^2+1$ subgroups of order $p$. This implies that $K_{p^2+1}$ is subgraph of $\overline{\mathscr{P}^*(G)}$. Therefore, $\gamma ({\overline{\mathscr{P}^*(G)}})>1$ and $\overline{\gamma} {(\overline{\mathscr{P}^*(G)})}>1$.
	
	\noindent\textbf{Subcase 2a(ii).} If $\overline{\mathscr{P}^*(G)}\cong \mathbb Z_{{p}^2}\times \mathbb Z_{p}$, then $G$ contains $p+1$ subgroups of order $p$, let them be $H_i$, $i=1,2,\ldots,p+1$. Also $G$ contains $p$ cyclic subgroups of order $p^2$, let them be $N_i$, $i=1,2,\ldots,p$. Moreover, all these subgroups contains the unique subgroups of order $p$, without loss of generality, let it be $H_1$. Then $\overline{\mathscr{P}^*(G)}$ contains $K_{p(p-1),p(p^2-1)}$ as a subgraph with the bipartition $X$ and $Y$, where $X$ contains all the elements of order $p$ in $H_i,$ $i=2,3,\ldots,p+1$; $Y$ contains all the non-identity elements in $N_i, i=1,2,\ldots,p$. This implies that $\overline{\mathscr{P}^*(G)}$ contains $K_{4,5}$, and so $\gamma ({\overline{\mathscr{P}^*(G)}})>1$, $\overline{\gamma} {(\overline{\mathscr{P}^*(G)})}>1$.
	
	\noindent\textbf{Subcase 2a(iii).}	If $G\cong (\mathbb Z_{p} \times \mathbb Z_p )\rtimes \mathbb Z_p$, then $G$ contains $p^2$ subgroups of order $p$. Then the subgraph of $\overline{\mathscr{P}^*(G)}$ induced by the set having   one element of order $p$ from each of these subgroups forms $K_{p^2}$, so $\gamma ({\overline{\mathscr{P}^*(G)}})>1$ and $\overline{\gamma} {(\overline{\mathscr{P}^*(G)})}>1$.
	
	\noindent\textbf{Subcase 2a(iv).} If $G\cong M_{p^3}$, then the subgroup lattice of $M_{p^3}$ is isomorphic to the subgroup lattice of $\mathbb Z_{p^2}\times \mathbb Z_p$, so by the above argument, we have $\gamma ({\overline{\mathscr{P}^*(G)}})>1$ and $\overline{\gamma} {(\overline{\mathscr{P}^*(G)})}>1$.
	
	\noindent\textbf{Subcase 2b.} If $p=2$, then upto isomorphism the only non-cyclic subgroup of order 8 are $\mathbb Z_2\times \mathbb Z_2\times \mathbb Z_2$, $\mathbb Z_{4}\times \mathbb Z_2$, $D_8$ and $Q_8$.
	
	\noindent\textbf{Subcase 2b(i).} If $G\cong \mathbb Z_2\times \mathbb Z_2\times \mathbb Z_2$, then the order of each element of $G$ is $2$. It follows that $\overline{\mathscr{P}^*(G)}\cong K_7$, so $\overline{\gamma} {(\overline{\mathscr{P}^*(G)})}>1$ and $\gamma ({\overline{\mathscr{P}^*(G)}})=1$.
	
	\noindent\textbf{Subcase 2b(ii).} If $G\cong \mathbb Z_{4}\times \mathbb Z_2$, then $G$ contains the elements $(1,0),(3,0),(1,1),(3,1)$ of order $4$, and the elements $(2,0),(0,1),(2,1) $ of order $2$. Here $(2,0)$ is a power of each of $(1,0),(3,0),(1,1),(3,1)$; $(3,0)$ is a power of $(1,0)$; $(3,1)$ is a power of $(1,1)$. Also no two remaining elements of $G$ are power of one another. Hence $\overline{\mathscr{P}^*(G)}$ contains $K_{3,3}$ with bipartition $X:=\left\{ (1,0),(3,0),(2,1)\right\}$ and $Y:=\left\{(1,1),(3,1),(0,1)\right\}$, so $\overline{\mathscr{P}^*(G)}$ is non-planar. Also  $\gamma ({\overline{\mathscr{P}^*(G)}})=1$, $\overline{\gamma} {(\overline{\mathscr{P}^*(G)})}=1$; a toroidal and a projective embedding of $\overline{\mathscr{P}^*(G)}$  is shown in Figure~\ref{Z4Z2} and~\ref{Z4.Z2}, respectively.
	\begin{figure}[ht]
		\begin{center}
			\includegraphics[scale=1]{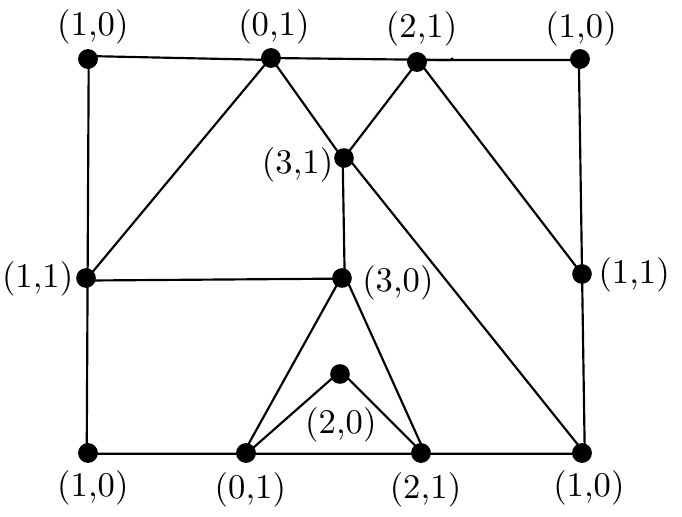}\caption{A toroidal embedding of $\overline{\mathscr{P}^*(\mathbb Z_{4}\times \mathbb Z_2)}.$} \label{Z4Z2}.
		\end{center}
	\end{figure}
	\begin{figure}[ht]
		\begin{center}
			\includegraphics[scale=1]{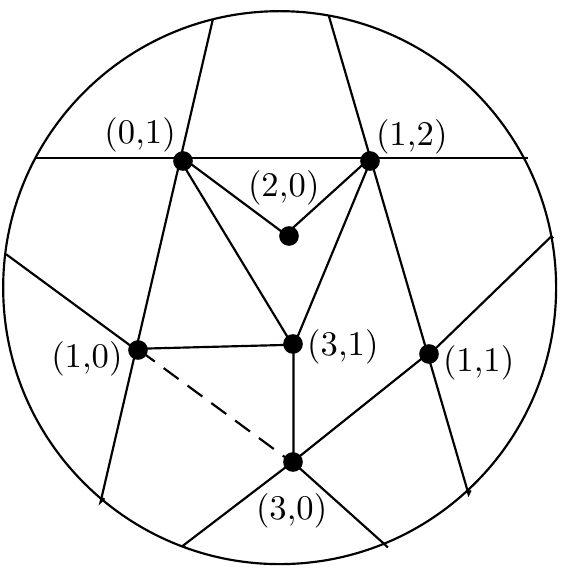}\caption{A projective embedding of $\overline{\mathscr{P}^*(\mathbb Z_{4}\times \mathbb Z_2)}.$} \label{Z4.Z2}
		\end{center}
	\end{figure}

	\noindent\textbf{Subcase 2b(iii).}	 If $G\cong D_8$, then $b$, $ab$, $a^2b$, $a^3b$, $a^2$ are the elements of order 2, and $a$, $a^3$ are the elements of order 4; these are the only elements of $D_8$. Here $\left\langle a\right\rangle =\left\langle a^3\right\rangle $ and it contains $a^2$, so $a$, $a^2$, $a^3$ are not adjacent to each other. Also any two remaining elements of $G$ are not a power of one another. Thus $a^2$, $b$, $ab$, $a^2b$, $a^3b$ forms $K_5$ as a subgraph of $\overline{\mathscr{P}^*(G)}$, so $\overline{\mathscr{P}^*(G)}$ non-planar. Further, $\gamma ({\overline{\mathscr{P}^*(G)}})=1$, $\overline{\gamma} {(\overline{\mathscr{P}^*(G)})}=1$; a toroidal and projective embedding of $\overline{\mathscr{P}^*(G)}$ is shown in Figure~\ref{M8} and \ref{M8P}.
	\begin{figure}[ht]
		\begin{center}
			\includegraphics[scale=1]{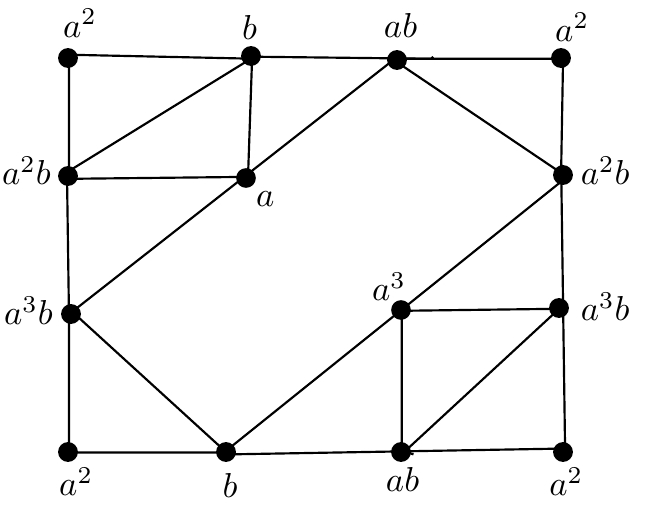}\caption{A toroidal embedding of $\overline{\mathscr{P}^*(D_8)}.$} \label{M8}
		\end{center}
	\end{figure}
	\begin{figure}[ht]
		\begin{center}
			\includegraphics[scale=1]{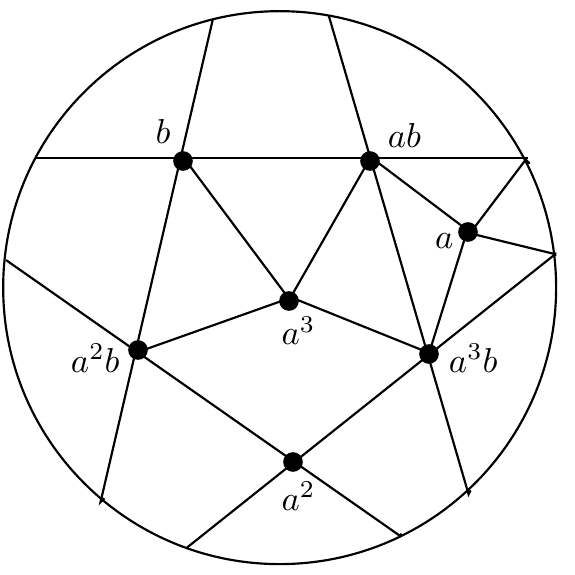}\caption{A projective embedding of $\overline{\mathscr{P}^*(D_8)}.$} \label{M8P}
		\end{center}
	\end{figure}
	
	\noindent\textbf{Subcase 2b(iv).} If $G\cong Q_8$, then by~ Figure~\ref{Q8}, $\overline{\mathscr{P}^*(G)}$  is planar.\\
	\noindent\textbf{Case 3.} Let $\alpha \geq 4$.\\
	\noindent\textbf{Subcase 3a.} Let $p=2$. If $\alpha=4$, then up to isomorphism there are four non-cyclic abelian groups of order $2^4$, and nine non-abelian groups of order $2^4$. In the following, first we deal with these non-cyclic abelian groups:
	
	If $G\cong \mathbb Z_4\times \mathbb Z_4$, then
	$G$ contains six cyclic subgroups, say $H_i$, $i= 1,2, \ldots , 6$ of order 4.  Hence $\overline{\mathscr{P}^*(G)}$ contains $K_{6,6}$ as a subgraph with the bipartition $X$, $Y$, where $X$ contains all the elements of order four in $H_1$, $H_2$ and $H_3$; $Y$ contains all the elements of order four in $H_4$, $H_5$ and $H_6$. Thus  $\overline{\gamma} {(\overline{\mathscr{P}^*(G)})}>1$ and $\gamma ({\overline{\mathscr{P}^*(G)}})>1$.
	
	If $G\cong \mathbb Z_8\times \mathbb Z_2$, then $G$ contains two cyclic subgroups of order 8, let them be $H_1$, $H_2$; two cyclic subgroups of order 4, let them be $N_1,N_2$ ; three elements of order 2, let them be $x_i$, $i=1,2,3$. Here $H_1$, $H_2$ contains a cyclic subgroup of order 4, and an element of order 2 in common, without loss of generality, let them be $N_1$ and $x_1$ respectively. So the elements of order 4 in $N_2$ is not a power of any non-trivial elements in $H_1$, $H_2$. Also $x_2$, $x_3$ are not elements of $H_1$, $H_2$, so they are not a power of any non-trivial elements in $H_1$, $H_2$ and vice versa. Hence $\overline{\mathscr{P}^*(G)}$ contains $K_{7,4}$ as a subgraph, and so $\gamma ({\overline{\mathscr{P}^*(G)}})>1$ and $\overline{\gamma} {(\overline{\mathscr{P}^*(G)})}>1$.
	
	If $G\cong \mathbb Z_4\times \mathbb Z_2 \times \mathbb Z_2$, then $G$ contains a subgroup isomorphic to $\mathbb Z_2\times \mathbb Z_2 \times \mathbb Z_2$. So $G$ contains seven elements of order 2; $(2,0,0)$ is one among these elements, and is a power of each of the elements $(1,0,0),(1,1,0),(1,0,1),(1,1,1)$ of order 4. But the remaining six elements of order 2 are not power of these four elements of order 4. Hence $\overline{\mathscr{P}^*(G)}$ contains $K_{5,6}$ as a subgraph with bipartition $X$ and $Y$, where $X=\left\lbrace (1,0,0),(1,1,0),(1,0,1),(1,1,1),(2,0,0) \right\rbrace $ and $Y$ contains the remaining six elements of order 2. Thus $\gamma ({\overline{\mathscr{P}^*(G)}})>1$ and $\overline{\gamma} {(\overline{\mathscr{P}^*(G)})}>1$.
	
	If $G\cong \mathbb Z_2\times \mathbb Z_2\times \mathbb Z_2 \times \mathbb Z_2$, then all the non-trivial elements of $G$ are of order 2 and so $\overline{\mathscr{P}^*(G)}\cong K_{11}$. Therefore, $\gamma ({\overline{\mathscr{P}^*(G)}})>1$ and $\overline{\gamma} {(\overline{\mathscr{P}^*(G)})}>1$.
	
	Next, we investigate the nine non-abelian groups of order $2^4$.
	
	If $G\cong (\mathbb Z_4\times \mathbb Z_2)\rtimes \mathbb Z_2$, then $G$ contains four cyclic subgroups of order 4, let them be $H_1$, $H_2$, $H_3$, $H_4$. Among these $H_1$, $H_2$ contains a unique element of order 2 in common, and $H_3$, $H_4$ contains a unique elements of order 2 in common. But $G$ contains exactly seven elements of order 2. So the remaining five elements in $G$ of order 2 are not a power of any non-trivial elements in $H_i$, $i=1,2,3,4$. Hence $\overline{\mathscr{P}^*(G)}$ contains $K_{5,8}$ as a subgraph, and so $\gamma ({\overline{\mathscr{P}^*(G)}})>1$ and $\overline{\gamma} {(\overline{\mathscr{P}^*(G)})}>1$.
	
	If $G\cong \mathbb Z_4\rtimes \mathbb Z_4$, then $G$ contains six cyclic subgroups of order 4. Let them be $H_i,$ $i=1,2,\ldots,6$. It follows that $\overline{\mathscr{P}^*(G)}$ contains $K_{6,4}$ with bipartition $X$ and $Y$, where $X$ contains all the elements of order four in $H_1$, $H_2$, $H_3$;  $Y$ contains all the elements of order four in $H_4$, $H_5$. Therefore $\gamma ({\overline{\mathscr{P}^*(G)}})>1$ and $\overline{\gamma} {(\overline{\mathscr{P}^*(G)})}>1$.
	
	If $G\cong \mathbb Z_8\rtimes_5 \mathbb Z_2$, then $G$ contains two cyclic subgroups of order 8, let them be $H_1$, $H_2$, these two cyclic subgroups contains a unique element of order 2 in common. But $G$ contains three elements of order 2,  so the remaining two elements of order 2 are not a power of non-trivial elements of $H_1$, $H_2$. It follows that $\overline{\mathscr{P}^*(G)}$ contains $K_{6,4}$ as a subgraph and so  $\gamma ({\overline{\mathscr{P}^*(G)}})>1$ and $\overline{\gamma} {(\overline{\mathscr{P}^*(G)})}>1$.
	
	If $G\cong D_{16}$, then $G$ contains nine elements of order 2, so $G$ contains $K_9$ as a subgraph. Therefore, $\gamma ({\overline{\mathscr{P}^*(G)}})>1$ and $\overline{\gamma} {(\overline{\mathscr{P}^*(G)})}>1$.
	
	If $G\cong \mathbb Z_8\rtimes_3 \mathbb Z_2$, then $G$ contains three subgroups of order 4, let them be $H_1$, $H_2$, $H_3$. These subgroups contains a unique element of order 2 in common. But $G$ contains five elements of order 2. So the remaining four elements of order 2 are not a power of any non-trivial elements in $H_i$, $i=1,2,3$. It follows that $G$ contains $K_{6,4}$ as a subgraph and so $\gamma ({\overline{\mathscr{P}^*(G)}})>1$ and $\overline{\gamma} {(\overline{\mathscr{P}^*(G)})}>1$.
	
	If $G\cong Q_{16}$, then $G$ contains five cyclic subgroup of order 4, let them be $H_i$, $i=1,\ldots,5$. Then $\overline{\mathscr{P}^*(G)}$ contains $K_{6,4}$ as a subgraph with bipartition $X$ and $Y$, where $X$ contains elements of order 4 in $H_1$, $H_2$, $H_3$, and $Y$ contains elements of order 4 in $H_4$, $H_5$. So $\gamma ({\overline{\mathscr{P}^*(G)}})>1$ and $\overline{\gamma} {(\overline{\mathscr{P}^*(G)})}>1$.
	
	If $G\cong D_8\times \mathbb Z_2$, then $G$ contains eleven elements of order 2. Hence they forms $K_{11}$ as a subgraph of $\overline{\mathscr{P}^*(G)}$ . It follows that $\gamma ({\overline{\mathscr{P}^*(G)}})>1$ and $\overline{\gamma} {(\overline{\mathscr{P}^*(G)})}>1$.
	
	If $G\cong Q_8\times \mathbb Z_2$, then $G$ contains six cyclic subgroups of order 4, let them be $H_i$, $i=1,\ldots,6$. Hence $\overline{\mathscr{P}^*(G)}$ contains $K_{6,6}$ as a subgraph with bipartition $X$ and $Y$, where $X$ contains elements of order 4 in $H_1$, $H_2$ and $H_3$, $Y$ contains elements of order four in $H_4$, $H_5$ and $H_6$. Thus, $\gamma ({\overline{\mathscr{P}^*(G)}})>1$ and $\overline{\gamma} {(\overline{\mathscr{P}^*(G)})}>1$.
	
	If $G\cong Q_8\rtimes \mathbb Z_2$, then $G$ contains seven elements of order 2 and four cyclic subgroups of order 4. Each of these cyclic subgroups contains exactly one element of order 2 in common, let it be $x$. Then $\overline{\mathscr{P}^*(G)}$ contains $K_{3,8}$ as a subgraph with bipartition $X$ and $Y$, where $X$ contains all the elements of order 2 in $G$ except $x$, and $Y$ contains all the elements of order 4 in $G$. So $\gamma ({\overline{\mathscr{P}^*(G)}})>1$ and $\overline{\gamma} {(\overline{\mathscr{P}^*(G)})}>1$.
	
	Assume that $\alpha\geq 5$. Then $G$ must contain a non-cyclic subgroup of order $2^{\alpha-1}$. For suppose all the subgroup of order $2^{\alpha-1}$ are cyclic, let  $H$, $K$ be two subgroups among these. Since $H$ is a subgroup of prime index, so $H$ is normal in $G$. It follows that $HK$  is a subgroup of $G$. If $|H\cap K|< 2^{\alpha-2}$, then $|HK|> |G|$, which is not possible. So $|H\cap K|$ must be $2^{\alpha-2}$. It follows that  $H$ and $K$  contains a common subgroup of order  $2^{\alpha-2}$. Hence $G$ has a unique subgroup of order $2^{\alpha-2}$. Then by Theorem~\ref{pre}(ii), $G$ must be cyclic, which is a contradiction to our hypothesis. Let this subgroup of $G$ of order $2^{\alpha-1}$ be $H$. Then by previous argument, $\gamma ({\overline{\mathscr{P}^*(H)}})>1$ and $\overline{\gamma} {(\overline{\mathscr{P}^*(H)})}>1$, and so $\gamma ({\overline{\mathscr{P}^*(G)}})>1$ and $\overline{\gamma} {(\overline{\mathscr{P}^*(G)})}>1$.\\
	\noindent\textbf{Subcase 3b.}
	Let $p\geq 3$. Then by Theorem~\ref{pre}(i), $G$ contains a non-cyclic subgroup $H$ of order $p^{\alpha-1}$. So by Case 2, $\gamma ({\overline{\mathscr{P}^*(H)}})>1$ and $\overline{\gamma} {(\overline{\mathscr{P}^*(H)})}>1$, and so $\gamma ({\overline{\mathscr{P}^*(G)}})>1$ and $\overline{\gamma} {(\overline{\mathscr{P}^*(G)})}>1$.
	
	Proof follows by combining all the cases together.
\end{proof}

Next, we consider the groups whose order has exactly two distinct prime factors.
First, we prove the following lemma, which is used later in the proof of Proposition \ref{p^mq^n}.
We consider the graph $K^{u,u'}_{3,3,3}$ obtained from the complete tripartite graph $K_{3,3,3}$ by adding two new vertices $u$ and $u'$ and six edges $uv_i$,  $u'v_i$, $i=1,2,3$, where each of the vertices $v_1,v_2,v_3$ comes from a different part of $K_{3,3,3}$, i.e. vertices $v_1,v_2,v_3$ induce a triangle in $K_{3,3,3}$.
Notice that the tripartite graph $K_{3,3,3}$ can also be considered as the complement of three vertex-disjoint triangles, which is denoted by $\overline{3K_3}$ in \cite{GKN2003}. We need to determine whether $K^{u,u'}_{3,3,3}$ is toroidal or not.

\begin{lemma}\label{K333}
	Graph $K^{u,u'}_{3,3,3}$ is non-toroidal.
\end{lemma}

\begin{proof} We use an approach similar to the proofs of Lemma 4.4 and Proposition 4.6 in \cite{GK2002}.
	First, we show that the graph $K_{3,3,3}=\overline{3K_3}$ has a unique embedding on the torus, which is a triangulation of the torus.
	This is shown computationally by using an exhaustive search method in \cite{GKN2003} (see also some corrected computational results for \cite{GKN2003} at \url{http://www.combinatorialmath.ca/G\&G/TorusMaps.html}).
	Then, we are going to show that it is not possible to extend this unique embedding of $K_{3,3,3}$ to an embedding of $K^{u,u'}_{3,3,3}$ on the torus.
	
	From the Euler's formula for the torus, we have $n-m+f=0$, where $n, m$, and $f$ are, respectively, the numbers of vertices, edges, and faces of a 2-cell embedding on the torus (e.g., see \cite{GKN2003}). Since $K_{3,3,3}$ has $n=9$ vertices and $m=27$ edges, an embedding of $K_{3,3,3}$ on the torus must have $18$ faces, i.e. $f=18$. Since each edge of an embedding appears either exactly once on the boundaries of two separate faces or two times on the boundary of the same face, and each face of an embedding is bounded by at least three edges, we have $3f\le 2m$. The equality $3f=2m$ is possible only when each face is a triangle, which is the case of $K_{3,3,3}$ on the torus. Therefore, an embedding of $K_{3,3,3}$ on the torus must be a triangulation.
	
	On the other hand, there are exactly two different embeddings of $K_{3,3}$ on the torus (e.g., see Figure 8 in \cite{GKN2003}), which are to be extended to an embedding of $K_{3,3,3}$ by adding into the 2-cell faces three new vertices adjacent to all six vertices of the original $K_{3,3}$. One of these embeddings of $K_{3,3}$ contains faces with only four vertices on their boundary: clearly, it is not possible to triangulate such a 4-vertex face by adding a new vertex adjacent to all six vertices of $K_{3,3}$. However, the other embedding of $K_{3,3}$ on the torus has each of its three faces containing all six vertices of $K_{3,3}$ on the face boundary. Adding a new vertex into each 6-vertex face and making it adjacent to all the vertices on the face boundary provides a triangulation of the torus by $K_{3,3,3}$. By this construction and symmetries of the embedding of $K_{3,3}$ with the hexagonal faces, the embedding of $K_{3,3,3}$ on the torus is unique.
	
	\begin{figure}[h!]
		\begin{center}
			\includegraphics[width=6cm]{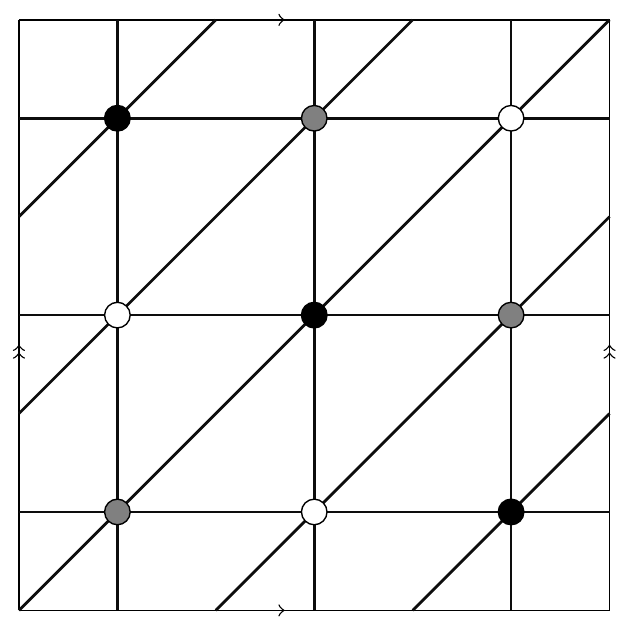}
			\caption{The unique embedding of $K_{3,3,3}$ on the torus.}
			\label{fig:K333}
		\end{center}
	\end{figure}
	
	Now, each face of the toroidal embedding of $K_{3,3,3}$ contains exactly three vertices, one from each part of $K_{3,3,3}$, no two faces have the same three vertices on the face boundary, and there are no vertices repeated on any face boundary (see Figure~\ref{fig:K333}).
	Each of the vertices $u$ or $u'$ of $K^{u,u'}_{3,3,3}$ separately can be added with its three incident edges into any of the $18$ faces of this embedding of $K_{3,3,3}$ without edge crossings.
	However, $u$ and $u'$ must be adjacent to the same three vertices on the boundary of a triangular face. There are no two different faces with the same three vertices on the boundary, and there are no faces with vertices repeated on the boundary.
	Therefore, after adding one of $u$ or $u'$ into a face of the embedding of $K_{3,3,3}$ and connecting it with edges to all three vertices on the face boundary, there is no face containing the same three vertices to add the other vertex without edge crossings.
	Thus, it is not possible to extend this unique embedding of $K_{3,3,3}$ to an embedding of $K^{u,u'}_{3,3,3}$ on the torus, and $K^{u,u'}_{3,3,3}$ is non-toroidal.
\end{proof}

\begin{pro}\label{p^mq^n}
	If $G$ is a non-cyclic group of order $p^nq^m$, where $p,~q$ are distinct primes and $n,m\geq 1$. Then
	\begin{enumerate}[{\normalfont (1)}]
		\item $\overline{\mathscr{P}^*(G)}$ is planar if and only if $G\cong S_3$;
		\item $\gamma {(\overline{\mathscr{P}^*(G)})}>1$ and  $\overline{\gamma} {(\overline{\mathscr{P}^*(G)})}>1$, if $G\ncong S_3$.	
	\end{enumerate}
\end{pro}
\begin{proof}
	In the proof of Proposition~\ref{cyclic}, we have noticed that, to prove this result, it is enough to consider the finite non-cyclic group $G$ of order
	$p^nq$, $n\geq 1$. \\
	\noindent\textbf{Case 1.} Let $n=1$. Without loss of generality, we assume that $p<q$, then $G\cong \mathbb Z_q\rtimes \mathbb Z_p$.
	If $(p,q)=(2,3)$, then $G\cong S_3$. By Figure~\ref{S3}, $\overline{\mathscr{P}^*(S_3)}$ is planar.
	If $(p,q)\neq (2,3)$, then by~(\ref{zqzp}), $\overline{\mathscr{P}^*(G)}$ contains $K_{5,4}$ as a subgraph, so $\gamma ({\overline{\mathscr{P}^*(G)}})>1$ and $\overline{\gamma} {(\overline{\mathscr{P}^*(G)})}>1$. \\
	\noindent\textbf{Case 2.} Let $n=2$. \\
	\noindent\textbf{Subcase 2a.} If $G$ is abelian, then $G\cong \mathbb Z_p\times \mathbb Z_{pq}$.
	
	First we assume that $p=2$. If $q=3$, the structure of $\overline{\mathscr{P}^*(G)}$ is shown in Figure~\ref{62}. Since this graph is isomorphic to the graph $K^{u,u'}_{3,3,3}$ described in Lemma~\ref{K333}, it follows that $\overline{\mathscr{P}^*(G)}$ is non-toroidal. Further, in Figure~\ref{62}, we notice that $K_{3,6}$ is a subgraph of  $\overline{\mathscr{P}^*(G)}$   and so  $\overline{\gamma} {(\overline{\mathscr{P}^*(G)})}>1$.
	\begin{figure}[ht]
		\begin{center}
			\includegraphics[scale=1]{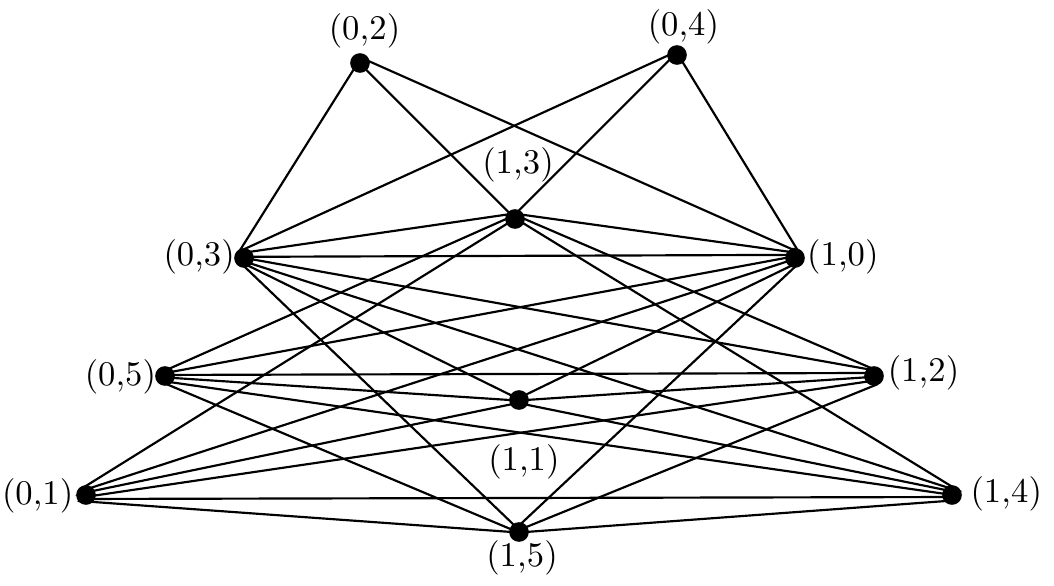}\caption{The structure of $\overline{\mathscr{P}^*(\mathbb Z_2\times \mathbb Z_6)}.$} \label{62}
		\end{center}
	\end{figure}	
	If $q=5$, then $G$ contains three cyclic subgroups of order $10$. This implies that $\overline{\mathscr{P}^*(G)}$ contains $K_{4,8}$, and so $\gamma ({\overline{\mathscr{P}^*(G)}})>1$ and $\overline{\gamma} {(\overline{\mathscr{P}^*(G)})}>1$. If $q\geq 7$, then $G$ contains at least two cyclic subgroups of order $q$. It follows that $\overline{\mathscr{P}^*(G)}$ contains $K_{6,6}$ as a subgraph, and so $\gamma ({\overline{\mathscr{P}^*(G)}})>1$ and $\overline{\gamma} {(\overline{\mathscr{P}^*(G)})}>1$.
	
	Next, we assume that $p=3$. Then $G$ contains $\mathbb Z_3\times \mathbb Z_3$ as a subgraph. So $G$ contains four cyclic subgroups of order $3$, let them be $H_1$, $H_2$, $H_3$, $H_4$. Let for each $i=1,2,3,4$,  $h_i$, $h_i'$ be the elements of $H_i$ of order $3$. Also $G$ contains an element of order $q$, say $x$. Then $\overline{\mathscr{P}^*(G)}$ contains $K_{5,4}$ as a subgraph with bipartition $X:=\left\{h_1,h_2,h_1',h_2',x \right\}$ and $Y:=\left\{h_3,h_4,h_3',h_4'\right\}$.
	
	If $p\geq 5$, then $G$ contains $\mathbb Z_p\times \mathbb Z_p$ as a subgraph, so by~(\ref{zpp}), $\gamma ({\overline{\mathscr{P}^*(G)}})>1$ and $\overline{\gamma} {(\overline{\mathscr{P}^*(G)})}>1$.\\
	\noindent\textbf{Subcase 2b.} Let $G$ be non-abelian.
	
	\noindent\textbf{Subcase 2b(i).} Let $p=2$.
	
	If $q=3$, then $G\cong \mathbb Z_3\rtimes \mathbb Z_4,~D_{12}$ or $A_4$. If $G\cong \mathbb Z_3\rtimes \mathbb Z_4$, then $G$ contains three cyclic subgroups, say $H_1$, $H_2$, $H_3$ of order 4;  unique cyclic subgroup $K$ of order $6$, and  unique element of order 2. It follows that $\overline{\mathscr{P}^*(G)}$ contains $K_{5,4}$ as a subgraph, so $\gamma ({\overline{\mathscr{P}^*(G)}})>1$ and $\overline{\gamma} {(\overline{\mathscr{P}^*(G)})}>1$. If $G\cong D_{12}$, then $G$ contains seven elements of order 2. These elements together with the  element of order 3 forms $K_8$ as a subgraph of $\overline{\mathscr{P}^*(G)}$, so $\gamma ({\overline{\mathscr{P}^*(G)}})>1$ and $\overline{\gamma} {(\overline{\mathscr{P}^*(G)})}>1$. If $G\cong A_4$, then $G$ contains eight elements of order 3, and three elements of order 2. It follows that $\overline{\mathscr{P}^*(G)}$ contains $K_{3,8}$ as a subgraph, so $\gamma ({\overline{\mathscr{P}^*(G)}})>1$ and $\overline{\gamma} {(\overline{\mathscr{P}^*(G)})}>1$.
	
	If $q=5$, then $G\cong \mathbb Z_5\rtimes \mathbb Z_4,~\mathbb Z_5\rtimes (\mathbb Z_2\times \mathbb Z_2)$ or $D_{20}$. If $G\cong \mathbb Z_5\rtimes \mathbb Z_4$, then $G$ contains five elements of order $2$, and four elements of order 5. This implies that $\overline{\mathscr{P}^*(G)}$ contains $K_{5,4}$ as a subgraph, so $\gamma ({\overline{\mathscr{P}^*(G)}})>1$ and $\overline{\gamma} {(\overline{\mathscr{P}^*(G)})}>1$. If $G\cong \mathbb Z_5\rtimes (\mathbb Z_2\times \mathbb Z_2)$, then $G$ contains four cyclic subgroups of order $4$, and a unique subgroup of order 5. Therefore,  $\overline{\mathscr{P}^*(G)}$ contains $K_{4,8}$, so $\gamma ({\overline{\mathscr{P}^*(G)}})>1$ and $\overline{\gamma} {(\overline{\mathscr{P}^*(G)})}>1$. If $G\cong D_{20}$, then $G$ contains eleven elements of order 2, and so $\overline{\mathscr{P}^*(G)}$ contains $K_{11}$ as a subgraph. Hence $\gamma ({\overline{\mathscr{P}^*(G)}})>1$ and $\overline{\gamma} {(\overline{\mathscr{P}^*(G)})}>1$.
	
	If $q=7$, then $G\cong \mathbb Z_7\rtimes \mathbb Z_4$ or $D_{28}$. If $G\cong \mathbb Z_7\rtimes \mathbb Z_4$, then $G$ contains seven cyclic subgroups of order 4, and six elements of order 7. It follows that  $\overline{\mathscr{P}^*(G)}$ contains $K_{5,4}$ as a subgraph. Thus $\gamma ({\overline{\mathscr{P}^*(G)}})>1$ and $\overline{\gamma} {(\overline{\mathscr{P}^*(G)})}>1$. If $G\cong D_{28}$, then $G$ contains fourteen elements of order 2 and so $\overline{\mathscr{P}^*(G)}$ contains $K_{5,4}$ as a subgraph. Therefore, $\gamma ({\overline{\mathscr{P}^*(G)}})>1$ and $\overline{\gamma} {(\overline{\mathscr{P}^*(G)})}>1$.
	
	If $q>7$, then $G$ contains at least two subgroups of order 4 and $q$, let them be $H_1$, $H_2$ respectively. Also each non-trivial element of $H_1$ is adjacent to all the non-trivial elements of $H_2$. It follows that $\overline{\mathscr{P}^*(G)}$ contains $K_{3,8}$ as a subgraph, so $\gamma ({\overline{\mathscr{P}^*(G)}})>1$ and $\overline{\gamma} {(\overline{\mathscr{P}^*(G)})}>1$.
	
	\noindent\textbf{Subcase 2b(ii).} Let $p=3$.
	
	(i) If $q=2$, then $G\cong S_3\times \mathbb Z_3$, $D_{18}$ or $\mathbb Z_3\times \mathbb Z_3\rtimes \mathbb Z_2$.
	
	If $G\cong S_3\times \mathbb Z_3$, then $G$ contains a subgroup of order 9, let it be $H$, and three elements of order 2. These three elements of order 2 are adjacent to all the non-trivial elements of $H$. This implies that $\overline{\mathscr{P}^*(G)}$ contains $K_{3,8}$ as a subgraph, and so $\gamma ({\overline{\mathscr{P}^*(G)}})>1$ and $\overline{\gamma} {(\overline{\mathscr{P}^*(G)})}>1$.
	
	If $G\cong D_{18}$ or $\mathbb Z_3\times \mathbb Z_3\rtimes \mathbb Z_2$, then in either case $\overline{\mathscr{P}^*(G)}$ contains nine elements of order 2. This implies that $\overline{\mathscr{P}^*(G)}$ contains $K_{9}$ as a subgraph, and so $\gamma ({\overline{\mathscr{P}^*(G)}})>1$ and $\overline{\gamma} {(\overline{\mathscr{P}^*(G)})}>1$.
	
(ii)	If $q\neq 2$, then $G$ contains a subgroup of each of order 9 and $q$, let them be $H_1$, $H_2$, respectively. Also every non-trivial element of $H_1$ are adjacent to all the non-trivial elements of $H_2$. It follows that $\overline{\mathscr{P}^*(G)}$ contains $K_{8,4}$ as a subgraph and so $\gamma ({\overline{\mathscr{P}^*(G)}})>1$ and $\overline{\gamma} {(\overline{\mathscr{P}^*(G)})}>1$.
	
	\noindent\textbf{Subcase 2b(iii).} Let $p\geq 5$.
	
	Then $G$ contains a subgroup of order $p^2$, let it be $H$. If $H$ is non-cyclic, then by~(\ref{zpp}), $\gamma ({\overline{\mathscr{P}^*(H)}})>1$ and $\overline{\gamma} {(\overline{\mathscr{P}^*(H))}}>1$, so $\gamma ({\overline{\mathscr{P}^*(G)}})>1$ and $\overline{\gamma} {(\overline{\mathscr{P}^*(G)})}>1$. Suppose $H$ is cyclic, then every element of $H$ of order $p^2$ are not a power of any element  which not in $H$ and vice versa. Thus $\overline{\mathscr{P}^*(G)}$ contains $K_{3,7}$ as a subgraph, and so $\gamma ({\overline{\mathscr{P}^*(G)}})>1$ and $\overline{\gamma} {(\overline{\mathscr{P}^*(G)})}>1$.\\
	\noindent\textbf{Case 3.} Let $n\geq 3$.\\
	\noindent\textbf{Subcase 3a.} If $G$ is abelian, then $G$ contains a subgroup isomorphic to $\mathbb Z_{pq}\times \mathbb Z_p\times \mathbb Z_p$ or $\mathbb Z_{pq}\times \mathbb Z_{p^2}$.
	
	If $G$ contains a subgroup isomorphic to $\mathbb Z_{pq}\times \mathbb Z_p\times \mathbb Z_p$, then $G$ contains $p^2+1$ cyclic subgroups of order $p$, let them be $H_i,$ $i=1,2,\ldots, p^2+1$. Also $G$ contains a cyclic subgroup of order $pq$, let it be $K$. Clearly $K$ contains a unique subgroup of order $p$, so without loss of generality, let it be $H_1$. Then $H_i$ $(i\neq 1)$ are not subgroups of $K$. So each non-trivial element of $K$ is adjacent to all the non-trivial elements in $H_i$ $(i\neq 1)$. It follows that $\overline{\mathscr{P}^*(G)}$ contains $K_{3,7}$ as a subgraph, so $\gamma ({\overline{\mathscr{P}^*(G)}})>1$ and $\overline{\gamma} {(\overline{\mathscr{P}^*(G)})}>1$.
	
	If $G$ contains a subgroup isomorphic to $\mathbb Z_{pq}\times \mathbb Z_{p^2}$, then $G$ contains $p+1$ cyclic subgroups  of order $p$, let them be $H_i$, $i=1,2,\ldots,p+1$, and two cyclic subgroups of order $p^2$, let them be $N_1$, $N_2$. These two subgroups contains a unique subgroup of order $p$ in common, so without loss of generality, let it be $H_1$. Then each non-trivial element of $N_1$, $N_2$ are adjacent to all the non-trivial elements in $H_i$ $(i\neq 1)$. Also $G$ contains a subgroup of order $pq$. It follows that $\overline{\mathscr{P}^*(G)}$ contains $K_{3,7}$ as a subgraph with bipartition $X$ and $Y$, where $X$ contains elements of order $p^2$ in $N_1$ and $N_2$; $Y$ contains elements of order $p$ in $H_i$ $(i\neq 1)$, the elements of order $q$, and $pq$ in $G$. So $\gamma ({\overline{\mathscr{P}^*(G)}})>1$ and $\overline{\gamma} {(\overline{\mathscr{P}^*(G)})}>1$.\\
	\noindent\textbf{Subcase 3b.} Let $G$ be non-abelian.\\
	\noindent\textbf{Subcase 3b(i).} Let $p=2$.
	
	(i) If $q=3$ and $n=3$, then  $G$ contains a subgroup of order 8, let it be $H$. But the only groups of  order 8 are $\mathbb Z_2\times \mathbb Z_2\times \mathbb Z_2$, $\mathbb Z_{4}\times \mathbb Z_2$, $D_8$, $Q_8$ and $\mathbb Z_8$.
	
	If $H\cong \mathbb Z_2\times \mathbb Z_2\times \mathbb Z_2$, then $H$ contains seven elements of order 2. These elements together with the element of order 3 forms $K_8$ as a subgraph of $\overline{\mathscr{P}^*(G)}$, so $\gamma ({\overline{\mathscr{P}^*(G)}})>1$ and $\overline{\gamma} {(\overline{\mathscr{P}^*(G)})}>1$.
	
	If $H\cong \mathbb Z_{4}\times \mathbb Z_2$, then $H$ contains two cyclic subgroups of order 4, let them be $H_1$, $H_2$, and three elements of order 2, let them be $x_1$, $x_2$, $x_3$. Here $H_1$, $H_2$ contains an element of order 2 in common, so without loss of generality, let  it be $x_1$. Then $x_2$, $x_3$ are adjacent to all the non-trivial elements of $H_1$ and $H_2$. Hence $\overline{\mathscr{P}^*(G)}$ contains $K_{5,4}$ as a subgraph with the bipartition $X$, $Y$, where $X$ contains all the non-trivial elements in $H_1$ and $H_2$; $Y$ contains $x_2$, $x_3$, and the elements of order 3 in $G$. Therefore, $\gamma ({\overline{\mathscr{P}^*(G)}})>1$ and $\overline{\gamma} {(\overline{\mathscr{P}^*(G)})}>1$.
	
	If $H\cong D_8$, then $H$ contains five elements of order 2, let them be $x_i$, $i=1,2,\ldots, 5$. Also $H$ contains a cyclic subgroup of order 4, let it be $H_1$, which contains only one element of order 2, let it be $x_1$. So $x_i$, $i\neq 1$ is adjacent to all the non-trivial elements in $H_1$. Hence $\overline{\mathscr{P}^*(G)}$ contains $K_{5,4}$ as a subgraph with the bipartition $X$, $Y$, where $X$ contains all the non-trivial elements in $H_1$, and the elements of order 3 in $G$; $Y$ contains only $x_i,$ $i\neq 1$. Therefore, $\gamma ({\overline{\mathscr{P}^*(G)}})>1$ and $\overline{\gamma} {(\overline{\mathscr{P}^*(G)})}>1$.
	
	If $H\cong Q_8$, then $G\cong Q_8\times \mathbb Z_3$ or $\mathbb Z_3 \rtimes Q_8$. If $G\cong Q_8\times \mathbb Z_3$, then $G$ contains three cyclic subgroups of order 12, let them be $H_1$, $H_2$, $H_3$.  $H$ contains three cyclic subgroup of order 4, let them be $H_1$, $H_2$ and $H_3$. Then $\overline{\mathscr{P}^*(G)}$ contains $K_{5,4}$ as a subgraph with the bipartition $X$, $Y$, where $X$ contains all the elements of order twelve in $H_1$, $H_2$; $Y$ contains all the elements of order twelve in $H_3$. therefore, $\gamma ({\overline{\mathscr{P}^*(G)}})>1$ and $\overline{\gamma} {(\overline{\mathscr{P}^*(G)})}>1$. If $G\cong \mathbb Z_3 \rtimes Q_8$, then $G$ contains seven elements of order 4, let them be $H_i$, $i=1,2,\ldots ,7$. So $\overline{\mathscr{P}^*(G)}$ contains $K_{5,4}$ as a subgraph with the bipartition $X$, $Y$, where $X$ contains all the elements of order four in $H_1$, $H_2$, $H_3$; $Y$ contains all the elements of order four in $H_4$, $H_5$. So, $\gamma ({\overline{\mathscr{P}^*(G)}})>1$ and $\overline{\gamma} {(\overline{\mathscr{P}^*(G)})}>1$. Moreover, these subgroups contains an element of order 2 in common, let it be $x$. Then $\overline{\mathscr{P}^*(G)}$ contains $K_{5,4}$ as a subgraph with the bipartition $X$, $Y$, where $X$ contains all non-trivial elements in $H_1$, $H_2$; $Y$ contains all the elements of order 3 in $G$, and elements of order 4 in $H_3$. So $\gamma ({\overline{\mathscr{P}^*(G)}})>1$ and $\overline{\gamma} {(\overline{\mathscr{P}^*(G)})}>1$.
	
	If $H\cong \mathbb Z_8$, then the elements not in $H$ are adjacent to all the non-trivial elements in $H$ of order $8$. Hence $\overline{\mathscr{P}^*(G)}$ contains $K_{7,3}$ as a subgraph. Therefore, $\gamma ({\overline{\mathscr{P}^*(G)}})>1$ and $\overline{\gamma} {(\overline{\mathscr{P}^*(G)})}>1$.
	
	If $n>3$, then $G$ contains a subgroup $H$ of order $2^n$. If $H$ is cyclic, then the element not in $H$ are adjacent to the elements in $H$ of order $2^{n}$, $n>3$. It follows that $\overline{\mathscr{P}^*(G)}$ contains $K_{7,3}$ as a subgraph. Therefore, $\gamma ({\overline{\mathscr{P}^*(G)}})>1$ and $\overline{\gamma} {(\overline{\mathscr{P}^*(G)})}>1$. If $H$ is non-cyclic, then by Proposition~\ref{p^m}, $\gamma ({\overline{\mathscr{P}^*(H)}})>1$ and $\overline{\gamma} {(\overline{\mathscr{P}^*(H)})}>1$, so $\gamma ({\overline{\mathscr{P}^*(G)}})>1$ and $\overline{\gamma} {(\overline{\mathscr{P}^*(G)})}>1$.
	
	(ii) If $q\geq 5 $, then $\overline{\mathscr{P}^*(G)}$ contains subgroups of order $2^{n-1}$ and $q$, let them be $H_1$ and $H_2$ respectively. Then each element in $H_1$ is not a power of any element in $H_2$, and vice versa. This implies that $\overline{\mathscr{P}^*(G)}$ contains $K_{7,3}$ as a subgraph and so $\gamma ({\overline{\mathscr{P}^*(G)}})>1$ and $\overline{\gamma} {(\overline{\mathscr{P}^*(G)})}>1$.
	
	
	\noindent\textbf{Subcase 3b(ii).} If $p\geq 3$, then $G$ contains a subgroup of order $p^n$, let it be $H$. If $H$ is non-cyclic, then by Theorem~\ref{p^m}, $\gamma ({\overline{\mathscr{P}^*(H)}})>1$ and $\overline{\gamma} {(\overline{\mathscr{P}^*(H)})}>1$, so $\gamma ({\overline{\mathscr{P}^*(G)}})>1$ and $\overline{\gamma} {(\overline{\mathscr{P}^*(G)})}>1$. If $H$ is cyclic, then the elements not in $H$ are adjacent to all the elements in $H$ of order $p^n$. It follows that $\overline{\mathscr{P}^*(G)}$ contains $K_{7,3}$ as a subgraph and so $\gamma ({\overline{\mathscr{P}^*(G)}})>1$ and $\overline{\gamma} {(\overline{\mathscr{P}^*(G)})}>1$.
\end{proof}

Proof of Theorem~\ref{planar} follows by combining all the propositions proved so far in this section.
\begin{proof}[Proof of Corollary~\ref{cor1}]
	Note that, if $\overline{\mathscr{P}^*(G)}$ is one of star, path, $C_n$, outerplanar, and not containing $K_{1,4}$ or $K_{2,3}$, then $\overline{\mathscr{P}^*(G)}$ must be planar. So to classify the finite groups whose complement of proper power graphs is one of these, it is enough to consider the finite group whose complement of proper power graph is planar. It is easy to check each of such possibilities among the list of groups given in Theorem~\ref{planar}(1), and their corresponding complement of proper power graph structure given in (\ref{p^n}), (\ref{pq}), (\ref{zpp}), (\ref{z22})  and Figures \ref{Z12}, \ref{Q8}, \ref{S3}. This completes the proof.	
\end{proof}

\section*{Acknowledgment}
This research work of the first author is supported  by Ministry of Social Justice $\&$ Empowerment and Ministry of Tribal Affairs, India in the form of  Rajiv Gandhi National Fellowship.

\end{document}